\documentclass{amsart}
\usepackage{amsmath}
\usepackage{amssymb}
\usepackage{mathrsfs}

\newtheorem{theorem}{Theorem}[section]
\newtheorem{claim}[theorem]{Claim}
\newtheorem{lemma}[theorem]{Lemma}

\newtheorem{conclusion}[theorem]{Conclusion}
\newtheorem{observation}[theorem]{Observation}

\theoremstyle{definition}
\newtheorem{definition}[theorem]{Definition}
\newtheorem{example}[theorem]{Example}

\newtheorem{discussion}[theorem]{Discussion}

\theoremstyle{remark}
\newtheorem{remark}[theorem]{Remark}

\newcommand{\frb}{{\rm frb}}
\newcommand{\com}{{\rm com}}
\newcommand{\frd}{{\rm frd}}

\newcommand{\bd}{{\rm bd}}
\newcommand{\tlim}{{\rm tlim}}
\newcommand{\nacc}{{\rm nacc}}
\newcommand{\acc}{{\rm acc}}
\newcommand{\prc}{{\rm prc}}
\newcommand{\nor}{{\rm nor}}
\newcommand{\prd}{{\rm prd}}
\newcommand{\pp}{{\rm pp}}
\newcommand{\Reg}{{\rm Reg}}
\newcommand{\RK}{{\rm RK}}
\newcommand{\ch}{{\rm ch}}
\newcommand{\Ch}{{\rm Ch}}

\newcommand{\otp}{{\rm otp}}

\newcommand{\Ord}{{\rm Ord}}

\newcommand{\cov}{{\rm cov}}
\newcommand{\tcf}{{\rm tcf}}
\newcommand{\pcf}{{\rm pcf}}

\newcommand{\Min}{{\rm Min}}
\newcommand{\Max}{{\rm Max}}
\newcommand{\Dom}{{\rm Dom}}
\newcommand{\Rang}{{\rm Rang}}

\newcommand{\wilog}{{\rm without loss of generality}}

\newcommand{\then}{{\underline{then}}}

\newcommand{\Then}{{\underline{Then}}}

\newcommand{\mn}{{\medskip\noindent}}
\newcommand{\sn}{{\smallskip\noindent}}

\newcommand{\cA}{{\mathscr A}}

\newcommand{\cB}{{\mathscr B}}

\newcommand{\gb}{{\mathfrak b}}
\newcommand{\ga}{{\mathfrak a}}
\newcommand{\gB}{{\mathfrak B}}
\newcommand{\gc}{{\mathfrak c}}
\newcommand{\gC}{{\mathfrak C}}

\newcommand{\gd}{{\mathfrak d\/}}

\newcommand{\cH}{{\mathscr H}}

\newcommand{\cF}{{\mathscr F}}

\newcommand{\cP}{{\mathscr P}}

\newcommand{\cT}{{\mathscr T}}

\newcommand{\cf}{{\rm cf}}

\newcount\skewfactor
\def\mathunderaccent#1#2 {\let\theaccent#1\skewfactor#2
\mathpalette\putaccentunder}
\def\putaccentunder#1#2{\oalign{$#1#2$\crcr\hidewidth
\vbox to.2ex{\hbox{$#1\skew\skewfactor\theaccent{}$}\vss}\hidewidth}}

\newenvironment{PROOF}[2][\proofname.]
   {\begin{proof}[#1]\renewcommand{\qedsymbol}{\textsquare$_{\rm #2}$}}
   {\end{proof}}

\begin{document}

\title {PCF: The advanced PCF theorems} 
 
\author {Saharon Shelah}
\address{Einstein Institute of Mathematics\\
Edmond J. Safra Campus, Givat Ram\\
The Hebrew University of Jerusalem\\
Jerusalem, 91904, Israel\\
 and \\
 Department of Mathematics\\
 Hill Center - Busch Campus \\
 Rutgers, The State University of New Jersey \\
 110 Frelinghuysen Road \\
 Piscataway, NJ 08854-8019 USA}
\email{shelah@math.huji.ac.il}
\urladdr{http://shelah.logic.at}
\thanks{The author thanks Alice Leonhardt for the beautiful typing.
I thank Peter Komjath for some comments. Paper E69}




\date{May 1, 2011}

\begin{abstract}
This is a revised version of \cite[\S6]{Sh:430}.
\end{abstract}

\maketitle
\numberwithin{equation}{section}
\newpage


\section {On pcf}

This is a revised version of \cite[\S6]{Sh:430} more self-contained,
large part done according to lectures in the Hebrew University Fall 2003

\noindent
Recall
\begin{definition}
\label{1.1}
Let $\bar f = \langle f_\alpha:\alpha < \delta \rangle,
f_\alpha \in {}^\kappa \Ord,I$ an ideal on $\kappa$.

\noindent
1) We say that $f \in {}^\kappa\Ord$ is a $\le_I$-l.u.b. of
$\bar f$ when:
\mn
\begin{enumerate}
\item[$(a)$]    $\alpha < \delta \Rightarrow f_\alpha \le_I f$
\sn
\item[$(b)$]    if $f' \in {}^\kappa \Ord$ and $(\forall
\alpha < \delta)(f_\alpha \le_I f')$ then $f \le_I f'$.
\end{enumerate}
\mn
2) We say that $f$ is a $\le_I$-e.u.b. of $\bar f$ when
\mn
\begin{enumerate}
\item[$(a)$]   $\alpha < \delta \Rightarrow f_\alpha \le_I f$
\sn
\item[$(b)$]   if $f' \in {}^\kappa \Ord$ and $f' <_I
\Max\{f,1_\kappa\}$ \then \, $f' <_I \Max
\{f_\alpha,1_\kappa\}$ for some $\alpha < \delta$.
\end{enumerate}
\mn
3) $\bar f$ is $\le_I$-increasing if $\alpha < \beta \Rightarrow
f_\alpha \le_I f_\beta$, similarly $<_I$-increasing.  We say $\bar
f$ is eventually $<_I$-increasing: it is $\le_I$-increasing and
$(\forall \alpha < \delta)(\exists \beta < \delta)(f_\alpha <_I f_\beta)$.

\noindent
4) We may replace $I$ by the dual ideal on $\kappa$.
\end{definition}

\begin{remark}
For $\kappa,I,\bar f$ as in Definition \ref{1.1},
if $\bar f$ is a $\le_I$-e.u.b. of $\bar f$ \then \, $f$ is a
$\le_I$-l.u.b. of $\bar f$.
\end{remark}

\begin{definition}
\label{1.2}
1) We say that $\bar s$ witness or exemplifies $\bar f$
is $(< \sigma)$-chaotic for $D$ when, for some $\kappa$
\mn
\begin{enumerate}
\item[$(a)$]   $\bar f = \langle f_\alpha:\alpha <\delta \rangle$ is
a sequence of members of ${}^\kappa \Ord$
\sn
\item[$(b)$]   $D$ is a filter on $\kappa$ (or an ideal on $\kappa$)
\sn
\item[$(c)$]   $\bar f$ is $<_D$-increasing
\sn
\item[$(d)$]   $\bar s = \langle s_i:i < \kappa \rangle,s_i$ a
non-empty set of $< \sigma$ ordinals
\sn
\item[$(e)$]   for every $\alpha < \delta$ for some $\beta \in
(\alpha,\delta)$ and $g \in \prod\limits_{i < \kappa} s_i$ we have
$f_\alpha \le_D g \le_D f_\beta$.
\end{enumerate}
\mn
2) Instead ``$(< \sigma^+)$-chaotic" we may say ``$\sigma$-chaotic".
\end{definition}

\begin{claim}
\label{1.3}
Assume
\mn
\begin{enumerate}
\item[$(a)$]  $I$ an ideal on $\kappa$
\sn
\item[$(b)$]  $\bar f = \langle f_\alpha:\alpha < \delta \rangle$
is $<_I$-increasing, $f_\alpha \in {}^\kappa \Ord$
\sn
\item[$(c)$]   $J \supseteq I$ is an ideal on $\kappa$
and $\bar s$ witnesses $\bar f$ is $(< \sigma)$-chaotic for $J$.
\end{enumerate}
\mn
\Then \, $\bar f$ has no $\le_I$-e.u.b. $f$ such that $\{i <
\kappa:\cf(f(i)) \ge \sigma\} \in J$.
\end{claim}

\begin{discussion}
\label{1.4}
What is the aim of clause (c) of \ref{1.3}?  For
$\le_I$-increasing sequence $\bar f,\langle f_\alpha:\alpha < \delta
\rangle$ in ${}^\kappa\Ord$ we are interested whether it has an
appropriate $\le_I$-e.u.b.  Of course, I may be a maximal ideal on
$\kappa$ and $\langle f_t:t \in \cf((\omega,<)^\kappa/D))$
is $<_I$-increasing cofinal in
$(\omega,<)^\kappa/D$, so it has an $<_I$-e.u.b. the sequence
$\omega_\kappa = \langle \omega:i <\kappa \rangle$, but this is not
what interests us now; we like to have a $\le_I$-e.u.b. $g$ such that
$(\forall i)(\cf(g(i)) > \kappa)$.
\end{discussion}

\begin{PROOF}{\ref{1.3}}
Toward contradiction assume that $f \in {}^\kappa \Ord$ is
a $\le_I$-e.u.b. of $\bar f$ and $A_1 := \{i < \kappa:\cf(f(i))
\ge \sigma\} \notin I$ hence $A \notin I$.

We define a function $f' \in {}^\kappa \Ord$ as follows:
\mn
\begin{enumerate}
\item[$\circledast$]    $(a) \quad$ if $i \in A$ then $f'(i) =
\sup(s_i \cap f(i)) +1$
\sn
\item[${{}}$]   $(b) \quad$ if $i \in \kappa \backslash A$ then
$f'(i) = 0$.
\end{enumerate}
\mn
Now that $i \in A \Rightarrow \cf(g(i)) \ge \sigma > |s_i|
\Rightarrow f'(i) < f(i) \le \Max\{g(i),1\}$ and $i \in \kappa
\backslash A \Rightarrow f'(i)=0 \Rightarrow f'(i) < \Max\{f(i),1\}$.
So by clause (b) of Definition \ref{1.1}(2) we
know that for some $\alpha < \delta$ we have $f' <_I
\Max\{f_\alpha,1\}$.
But ``$\bar s$ witness that $\bar f$ is $(< \sigma)$-chaotic"
hence we can find $g \in \prod\limits_{i < \kappa} s_i$ and
$\beta \in (\alpha,\delta)$ such that $f_\alpha \le_I g \le_I
f_\beta$ and as $\bar f$ is $<_I$-increasing \wilog \, $g <_I f_\beta$.

So $A_2 := \{i < \kappa:f_\alpha(i) \le g(i) < f_\beta(i) \le f(i)$
and $f'(i) < \Max\{f_\alpha(i),1\} = \kappa\} \mod I$ hence $A :=
A_1 \cap A_2 \ne \emptyset \mod I$ hence $A \ne \emptyset$.  So for
any $i \in A$ we have $f_\alpha(i) \le g(i) < f_\beta(i) \le f(i)$ and
$f(i) \in s_i$ hence $g(i) < f'(i) := \sup(s_i \cap f(i))+1$ and so
$f'(i) \ge 1$.

Also $f'(i) < \Max\{f_\alpha(i,1)\}$ hence $f'(i) < f_\alpha(i)$.
Together $f'(i) < f_\alpha(i) \le g(i) < f'(i)$, contradiction.
\end{PROOF}

\begin{lemma}
\label{1.5}
Suppose $\cf(\delta) > \kappa^+,I$ an ideal on $\kappa$ and
$f_{\alpha} \in {}^\kappa \Ord$ for $\alpha <
\delta$ is $\le_{I}$-increasing.  \Then \, there are $\bar J,\bar s,
\bar f'$ satisfying:
\mn
\begin{enumerate}
\item[$(A)$]    $\bar s = \langle s_i:i < \kappa \rangle$,
each $s_{i}$ a set of  $\le \kappa$ ordinals,
\sn
\item[$(B)$]    $\sup\{f_\alpha(i):\alpha < \delta\} \in s_i$;
moreover is $\max(s_i)$
\sn
\item[$(C)$]    $\bar f' = \langle f'_\alpha:\alpha <\delta \rangle$
where $f'_{\alpha} \in \prod\limits_{i<\kappa} s_{i}$ is defined by
$f'_{\alpha}(i) = \Min\{s_{i}\backslash f_{\alpha}(i)\}$,
(similar to rounding!)
\sn
\item[$(D)$]    $\cf[{f'_\alpha}(i)] \le \kappa$ (e.g.
$f'_\alpha(i)$ is a successor ordinal) implies $f'_\alpha(i) = f_\alpha(i)$
\sn
\item[$(E)$]    $\bar J = \langle J_\alpha:\alpha < \delta \rangle,
J_{\alpha}$ is an ideal on $\kappa$ extending $I$ (for $\alpha <
\delta$), decreasing with $\alpha$ (in fact for some $a_{\alpha,\beta}
\subseteq \kappa$ (for $\alpha < \beta < \kappa$) we have
$a_{\alpha,\beta}/I$  decreases with  $\beta$,  increases with
$\alpha$ and $J_{\alpha}$ is the ideal generated by
$I \cup \{a_{\alpha,\beta}:\beta$ belongs to $(\alpha,\lambda)\}$) so
possibly $J_\alpha = \cP(\kappa)$ and possibly $J_\alpha=I$
\end{enumerate}
\mn
such that:
\mn
\begin{enumerate}
\item[$(F)$]   if $D$ is an
ultrafilter on $\kappa$
disjoint to $J_{\alpha}$ \then \, $f'_{\alpha}/D$ is a $<_{D}$-l.u.b and
even $<_D$-e.u.b. of $\langle f_{\beta}/D:\beta < \alpha\rangle$
which is eventually $<_D$-increasing and
$\{i < \kappa:\cf[f'_{\alpha}(i))] > \kappa\} \in D$.
\end{enumerate}
\mn
Moreover
\mn
\begin{enumerate}
\item[$(F)^+$]    if $\kappa \notin J_{\alpha}$ \then \, $f'_{\alpha}$
is an $<_{J_\alpha}$-e.u.b (= exact upper bound) of $\langle
f_{\beta}:\beta < \delta \rangle$ and $\beta \in (\alpha,\delta)
\Rightarrow f'_\beta =_{J_\alpha} f'_\alpha$
\sn
\item[$(G)$]    if $D$ is an ultrafilter on $\kappa$ disjoint to $I$
but for every $\alpha$ not disjoint to $J_{\alpha}$ \then \,
$\bar s$ exemplifies $\langle f_{\alpha}:\alpha < \delta \rangle$ is $\kappa$
chaotic for $D$ as exemplified by $\bar s$ (see Definition \ref{1.2}), i.e.,
for some club $E$  of  $\delta,\beta < \gamma \in E \Rightarrow
f_{\beta} \le_D f'_{\beta} <_D f_{\gamma}$
\sn
\item[$(H)$]    if $\cf(\delta)> 2^\kappa$ \then \,
$\langle f_{\alpha}:\alpha < \delta \rangle$ has a
$\le_{I}$-l.u.b. and even $\le_I$-e.u.b. and for every large enough
$\alpha$ we have $I_\alpha = I$
\sn
\item[$(I)$]    if $b_\alpha =:\{i:{f'_\alpha}(i)$ has cofinality
$\le \kappa$ (e.g., is a successor)$\} \notin J_\alpha$ \then \,: for
every $\beta\in(\alpha, \delta)$ we have
$f'_\alpha \upharpoonright b_\alpha= f_\beta \upharpoonright
b_\alpha \mod J_\alpha$.
\end{enumerate}
\end{lemma}

\begin{remark}
Compare with \cite{Sh:506}.
\end{remark}

\begin{PROOF}{\ref{1.5}}
Let $\alpha^* = \cup\{f_\alpha(i) +1:\alpha < \delta,i
<\kappa\}$ and $S = \{j < \alpha^*:j$ has cofinality $\le \kappa\},
\bar e = \langle e_{j}:j \in S \rangle$ be such that
\mn
\begin{enumerate}
\item[$(a)$]   $e_j \subseteq j,|e_j| \le \kappa$ for every $j \in S$
\sn
\item[$(b)$]   if $j=i+1$ then $e_j = \{i\}$
\sn
\item[$(c)$]   if $j$ is limit, then $j = \sup(e_j)$ and $j' \in S
\cap e_j \Rightarrow e_{j'} \subseteq e_j$.
\end{enumerate}
\mn
For a set $a \subseteq \alpha^*$ let $c \ell_{\bar e}(a) =
a \cup \bigcup\limits_{j \in a \cap S} e_j$ hence by clause (c) clearly
$c \ell_{\bar e}(c \ell_{\bar e}(a)) = c \ell_{\bar e}(a)$ and
$[a \subseteq b \Rightarrow c \ell_{\bar e}(a) \subseteq
c \ell_{\bar e}(b)]$ and
$|c \ell_{\bar e}(a)| \le |a|+\kappa$.  We try to choose by
induction on $\zeta < \kappa^+$, the following objects:  $\alpha_{\zeta}$,
$D_{\zeta}$, $g_{\zeta}$, $\bar s_{\zeta} = \langle s_{\zeta,i}:i <
\kappa \rangle,\langle f_{\zeta,\alpha}:\alpha < \delta\rangle$ such that:
\mn
\begin{enumerate}
\item[$\boxtimes$]   $(a) \quad g_{\zeta} \in {}^\kappa \Ord$ and
$g_\zeta(i) \le \cup\{f_\alpha(i):\alpha < \delta\}$
\sn
\item[${{}}$]   $(b) \quad s_{\zeta,i} = c \ell_{\bar e}
[\{g_{\epsilon}(i):\epsilon < \zeta\} \cup \{\sup_{\alpha <\delta}
f_{\alpha}(i)\}]$ so it is a set of $\le \kappa$  ordinals

\hskip25pt increasing with $\zeta$ and
$\sup_{\alpha <\delta} f_{\alpha}(i) \in s_{\zeta,i}$,

\hskip25pt moreover $\sup_{\alpha < \delta} f_\alpha(i) = \max(s_{\zeta,i})$
\sn
\item[${{}}$]   $(c) \quad f_{\zeta,\alpha} \in {}^\kappa \Ord$ is
defined by $f_{\zeta,\alpha}(i) = \Min\{s_{\zeta,i} \backslash
f_{\alpha}(i)\}$,
\sn
\item[${{}}$]   $(d) \quad D_{\zeta}$ is an
ultrafilter on $\kappa$ disjoint to $I$
\sn
\item[${{}}$]   $(e) \quad f_{\alpha} \le_{D_\zeta} g_{\zeta}$
for $\alpha < \delta$
\sn
\item[${{}}$]   $(f) \quad \alpha_{\zeta}$ is an ordinal  $< \delta$
\sn
\item[${{}}$]    $(g) \quad \alpha_{\zeta} \le \alpha < \delta \Rightarrow
g_{\zeta} <_{D_\zeta} f_{\zeta,\alpha}$.
\end{enumerate}
\mn
If we succeed, let $\alpha(*) =  \sup\{\alpha_\zeta:\zeta <
\kappa^+\}$,  so as $\cf(\delta)> \kappa^+$ clearly
$\alpha(*) < \delta$.  Now let $i < \kappa$  and look at
$\langle f_{\zeta,\alpha(*)}(i):\zeta < \kappa^+\rangle$;
by its definition (see clause (c)), $f_{\zeta,\alpha(*)}(i)$  is the
minimal member of the set  $s_{\zeta,i} \backslash f_{\alpha(*)}(i)$.
This set increases with $\zeta$, so $f_{\zeta,\alpha(*)}(i)$
decreases with  $\zeta$ (though not necessarily strictly), hence is
eventually constant; so for some  $\xi_{i} < \kappa^+$ we have
$\zeta \in [\xi_{i},\kappa^+) \Rightarrow f_{\zeta,\alpha(*)}(i)
= f_{\xi_i,\alpha(*)}(i)$.
Let $\xi(*) = \sup_{i<\kappa} \xi_{i}$, so $\xi(*) < \kappa^+$, hence
\mn
\begin{enumerate}
\item[$\bigodot_1$]    $\zeta \in [\xi(*),\kappa^+)\; \, \and \;  i <
\kappa \Rightarrow f_{\zeta,\alpha(*)}(i) = f_{\xi(*),\alpha(*)}(i)$.
\end{enumerate}
\mn
By clauses (e) + (g) of $\boxtimes$ we know that
$f_{\alpha(*)} \le_{D_{\xi(*)}} g_{\xi(*)}
<_{D_{\xi(*)}} f_{\xi(*),\alpha(*)}$ hence
for some $i < \kappa$ we have $f_{\alpha(*)}(i) \le g_{\xi(*)}(i) <
f_{\xi(*),\alpha(*)}(i)$.  But $g_{\xi(*)}(i) \in
s_{\xi(*)+1,i}$ by clause (b) of $\boxtimes$
hence recalling the definition of
$f_{\xi(*)+1,\alpha(*)}(i)$ in clause (c) of $\boxtimes$
and the previous sentence $f_{\xi(*)+1,\alpha(*)}(i)
\le g_{\xi(*)}(i) < f_{\xi(*),\alpha(*)}(i)$,
contradicting the statement $\odot_1$.

So necessarily we are stuck in the induction process.  Let $\zeta <
\kappa^+$ be the first ordinal that breaks the induction.
Clearly $s_{\zeta,i}(i < \kappa),f_{\zeta,\alpha}(\alpha < \delta)$
are well defined.

Let $s_{i} =: s_{\zeta,i}$ (for  $i < \kappa )$  and $f'_{\alpha}  =
f_{\zeta,\alpha}$ (for $\alpha < \delta$), as defined in $\boxtimes$,
clearly they are well defined.
Clearly  $s_{i}$ is a set of $\leq\kappa$ ordinals and:
\mn
\begin{enumerate}
\item[$(*)_1$]     $f_{\alpha} \le f'_{\alpha}$
\sn
\item[$(*)_2$]    $\alpha < \beta \Rightarrow f'_{\alpha} \le_I f'_{\beta}$
\sn
\item[$(*)_3$]    if $b = \{i:f'_\alpha(i) < f'_{\beta}(i)\}
\notin I$ and $\alpha < \beta < \delta$ then $f'_\alpha
\upharpoonright b <_I f_{\beta} \upharpoonright b$.
\end{enumerate}
\mn
We let for $\alpha < \delta$
\mn
\begin{enumerate}
\item[$\bigodot_2$]    $J_{\alpha} = \bigl\{ b \subseteq  \kappa:b
\in I \text{ or } b \notin I \text{ and for every } \beta \in
(\alpha,\delta) \text{ we have}$:

\hskip40pt $f'_{\alpha} \upharpoonright (\kappa \setminus b) =_I f'_{\beta}
\upharpoonright (\kappa\setminus b) \bigr\}$
\sn
\item[$\bigodot_3$]    for $\alpha < \beta < \delta$ we let
$a_{\alpha,\beta}=:\{i<\kappa:f'_\alpha(i) <
f'_\beta(i)\}$.
\end{enumerate}
\mn
\Then \, as $\langle f'_{\alpha}:\alpha < \delta \rangle$ is
$\le_{I}$-increasing (i.e., $(*)_2$):
\mn
\begin{enumerate}
\item[$(*)_4$]   $a_{\alpha,\beta}/I$ increases with $\beta$,
decreases with $\alpha$, $J_{\alpha}$ increases with $\alpha$
\sn
\item[$(*)_5$]    $J_{\alpha}$ is an ideal on $\kappa$ extending
$I$, in fact is the ideal generated by $I \cup\{a_{\alpha,\beta}:
\beta\in(\alpha,\delta)\}$
\sn
\item[$(*)_6$]   if $D$ is an ultrafilter on  $\kappa$ disjoint to
$J_{\alpha}$, then  $f'_{\alpha} /D$  is a $<_{D}$-lub of
$\{f_{\beta}/D:\beta < \delta\}$.
\end{enumerate}
\mn
[Why?  We know that $\beta \in(\alpha,\delta) \Rightarrow
a_{\alpha,\beta} = \emptyset \mod D$,  so $f_{\beta} \le
f'_{\beta} =_D f'_{\alpha}$ for  $\beta \in (\alpha,\delta)$,
so  $f'_{\alpha} /D$  is an $\le_{D}$-upper bound.
If it is not a least upper bound then for some $g \in {}^\kappa \Ord$,
for every $\beta < \delta$ we have
$f_{\beta} \le_{D} g <_{D} f'_{\alpha}$
and we can get a contradiction to the choice of
$\zeta,\bar s,f'_{\beta}$ because: $(D,g,\alpha)$ could serve as
$D_{\zeta},g_{\zeta},\alpha_\zeta$.]
\mn
\begin{enumerate}
\item[$(*)_7$]    If $D$  is an ultrafilter on $\kappa$ disjoint
to $I$ but not to $J_{\alpha}$ for every $\alpha < \delta$ then $\bar s$
exemplifies that $\langle f_{\alpha}:\alpha < \delta \rangle$ is
$\kappa^+$-chaotic for  $D$, see Definition \ref{1.2}.
\end{enumerate}
\mn
[Why?  For every  $\alpha < \delta $  for some  $\beta \in
(\alpha,\delta)$  we have $a_{\alpha,\beta}\in D$, i.e., $\{i <
\kappa:f'_{\alpha} (i) < f'_{\beta}(i)\} \in D$,  so
$\langle f'_{\alpha}/D:\alpha < \delta \rangle$
is not eventually constant, so if  $\alpha < \beta,f'_{\alpha} <_{D}
f'_{\beta}$  then  $f'_{\alpha} <_{D} f_{\beta}$ (by $(*)_3$) and
$f_\alpha \le_{D} f'_\alpha$ (by (c)).  So $f_\alpha \le_D f'_\alpha
<_D f_\beta$ as required.]
\mn
\begin{enumerate}
\item[$(*)_8$]    if $\kappa \notin J_\alpha$ then $f'_\alpha$ is an
$\le_{J_\alpha}$-e.u.b. of $\langle f_\beta:\beta<\delta\rangle$.
\end{enumerate}
\mn
[Why?  By $(*)_6$, $f'_\alpha$ is a $\le_{J_\alpha}$-upper bound
of $\langle f_\beta:\beta<\delta\rangle$; so assume that it is
not a $\le_{J_\alpha}$-e.u.b. of $\langle
f_\beta:\beta<\delta\rangle$, hence there is a function $g$ with
domain $\kappa$, such that $g <_{J_\alpha} \Max \{1,f'_\alpha\}$,
but for no $\beta<\delta$ do we have

\[
c_\beta =:\{ i<\kappa:g(i) < \Max\{1,f_\beta(i)\}\} = \kappa \mod J_\alpha.
\]

\mn
Clearly $\langle c_\beta:\beta<\delta\rangle$ is increasing
modulo $J_\alpha$ so there is an ultrafilter $D$ on $\kappa$
disjoint to $J_\alpha \cup\{c_\beta:\beta<\delta\}$.
So $\beta < \delta \Rightarrow
f_\beta \le_D g \le_D f'_\alpha$, so we get a contradiction to
$(*)_6$ except when $g =_D f'_\alpha $ and then
$f'_\alpha =_D 0_\kappa$ (as $g(i)< 1 \vee  g(i)< f'_\alpha(i)$).
If we can demand $c^* = \{i:f'_\alpha(i) =0\}\notin D$ we are
done, but easily $c^* \setminus c_\beta\in J_\alpha$ so we finish.]

\mn
\begin{enumerate}
\item[$(*)_9$]    If $\cf[f'_\alpha(i)]\le\kappa$ then $f'_\alpha(i)
 = f_\alpha(i)$ so clause (D) of the lemma holds.
\end{enumerate}
\mn
[Why?  By the definition of $s_\zeta = c \ell_{\bar e}[\ldots]$ and
the choice of $\bar e$, and of $f'_\alpha(i)$.]
\mn
\begin{enumerate}
\item[$(*)_{10}$]    Clause (I) of the conclusion holds.
\end{enumerate}
\mn
[Why?  As $f_\alpha \le_{J_\alpha} f_\beta \le_{J_\alpha} f'_\alpha$ and
$f_\alpha \upharpoonright b_\alpha =_{J_\alpha} f'_\alpha
\upharpoonright b_\alpha$ by $(*)_9$.]
\mn
\begin{enumerate}
\item[$(*)_{11}$]   if $\alpha < \beta < \delta$ then $f'_\alpha =
f'_\beta \mod J_\alpha$, so clause (F)$^+$ holds.
\end{enumerate}
\mn
[Why?  First, $\bar f$ is $\le_I$-increasing hence it is
$\le_{J_\alpha}$-increasing.  Second, $\beta \le \alpha \Rightarrow
f_\beta \le_I f_\alpha \le f'_\alpha \Rightarrow f_\beta
\le_{J_\alpha} f'_\alpha$.  Third, if $\beta \in (\alpha,\delta)$ then
$a_{\alpha,\beta} = \{i < \kappa:f'_\alpha(i) < f'_\beta(i)\} \in
J_\alpha$, hence $f'_\beta \le_{J_\alpha} f'_\alpha$ but as $f_\alpha
\le_I f_\beta$ clearly $f'_\alpha \le_I f'_\beta$ hence $f'_\alpha
\le_{J_\alpha} f'_\beta$, so together $f'_\alpha =_{J_\alpha}
f'_\beta$.]
\mn
\begin{enumerate}
\item[$(*)_{12}$]   if $\cf(\delta) > 2^\kappa$ then for some $\alpha(*),
J_{\alpha(*)} = I$ (hence $\bar f$ has a $\le_I$-e.u.b.)
\end{enumerate}
\mn
[Why?  As $\langle J_\alpha:\alpha < \delta \rangle$ is a
$\subseteq$-decreasing sequence of subsets of $\cP(\kappa)$ it is
eventually constant, say, i.e., there is $\alpha(*) < \delta$ such
that $\alpha(*) \le \alpha < \delta \Rightarrow J_\alpha =
J_{\alpha(*)}$.  Also $I \subseteq J_{\alpha(*)}$, but if $I \ne
J_{\alpha(*)}$ then there is an ultrafilter $D$ of $\kappa$ disjoint
to $I$ but not to $J_{\alpha(*)}$ hence $\langle s_i:i < \kappa
\rangle$ witness being $\kappa$-chaotic.  But this implies $\cf(\delta)
\le \prod\limits_{i < \kappa} |s_i| \le \kappa^\kappa = 2^\kappa$,
contradiction.]

The reader can check the rest.
\end{PROOF}

\begin{example}
\label{1.6}
1) We show that l.u.b and e.u.b are not
the same.  Let $I$ be an ideal on  $\kappa,\kappa^+ < \lambda =
\cf(\lambda),\bar a = \langle a_{\alpha}:\alpha < \lambda \rangle $  be
a sequence of subsets of  $\kappa $,  (strictly) increasing modulo $I$,
$\kappa \backslash a_{\alpha}  \notin  I$  but there is no
$b \in \cP(\kappa)\backslash I$  such that $\bigwedge\limits_{\alpha}
 b \cap  a_{\alpha}  \in  I$.
[Does this occur?  E.g., for  $I = [\kappa]^{< \kappa}$,
the existence of such  $\bar a$  is known to be consistent; e.g.,
MA $\;  \and \;  \kappa =\aleph_0 \;\;   \and\;  \,  \lambda = 2^{\aleph_{0}}$.
Moreover, for any  $\kappa $  and $\kappa^+ < \lambda = \cf(\lambda)
\le 2^\kappa $ we can find $a_{\alpha} \subseteq \kappa$ for
$\alpha < \lambda $  such that, e.g., any Boolean combination of the
$a_{\alpha}$'s has cardinality $\kappa$ (less needed).
Let $I_{0}$ be the ideal on $\kappa$ generated by
$[\kappa]^{< \kappa} \cup \{a_{\alpha} \backslash a_{\beta}:
\alpha < \beta < \lambda\}$,  and let $I$ be maximal in
$\{J:J$ an ideal on $\kappa,I_0 \subseteq J$ and
$[\alpha < \beta < \lambda \Rightarrow a_{\beta} \backslash a_{\alpha}
\notin J]\}$.   So if G.C.H. fails, we have examples.]

For $\alpha < \lambda$,  we let $f_{\alpha}:\kappa \rightarrow \Ord$ be:

\[
f_{\alpha}(i) = \begin{cases} \alpha &\text{ if } i \in \kappa \setminus
a_\alpha, \\
\lambda + \alpha &\text{ if } i \in a_\alpha.
\end{cases}
\]

\mn
Now the constant function $f \in {}^\kappa \Ord,f(i) = \lambda +
\lambda$  is a l.u.b of $\langle f_{\alpha}:\alpha < \lambda \rangle$
but not an e.u.b. (both $\mod I$) (no e.u.b. is exemplified by $g \in
{}^\kappa \Ord$ which is constantly $\lambda$).

\noindent
2) Why do we require ``$\cf(\delta) > \kappa^+$" rather than
``$\cf(\delta) > \kappa$"?  As we have to, by Kojman-Shelah \cite{KjSh:673}.
\end{example}

\noindent
Recall (see \cite[2.3(2)]{Sh:506})
\begin{definition}
\label{1.7}
We say that $\bar f = \langle f_\alpha:\alpha < \delta \rangle$
obeys $\langle u_\alpha:\alpha \in S \rangle$ when
\mn
\begin{enumerate}
\item[$(a)$]   $f_\alpha:w \rightarrow \Ord$ for some fixed set $w$
\sn
\item[$(b)$]    $S$ a set of ordinals
\sn
\item[$(c)$]  $u_\alpha \subseteq \alpha$
\sn
\item[$(d)$]   if $\alpha \in S \cap \delta$ and $\beta \in u_\alpha$ then
$t \in w \Rightarrow f_\beta(t) \le f_\alpha(t)$.
\end{enumerate}
\end{definition}

\begin{claim}
\label{1.8}
Assume $I$ is an ideal on $\kappa,\bar f =
\langle f_\alpha:\alpha < \delta \rangle$ is $\le_I$-increasing and
obeys $\bar u = \langle u_\alpha:\alpha \in S \rangle$.  The sequence
$\bar f$ has a $\le_I$-e.u.b. when for some $S^+$ we have
$\circledast_1$ or $\circledast_2$ where
\mn
\begin{enumerate}
\item[$\circledast_1$]   $(a) \quad S^+ \subseteq \{\alpha <
\delta:\cf(\alpha) > \kappa\}$
\sn
\item[${{}}$]   $(b) \quad S^+$ is a stationary subset of $\delta$
\sn
\item[${{}}$]   $(c) \quad$ for each $\alpha \in S^+$ there are
unbounded subsets $u,v$ of $\alpha$ for which

\hskip25pt $\beta \in v \Rightarrow u \cap\beta \subseteq u_\beta$.
\sn
\item[$\circledast_2$]   $S^+ = \{\delta\}$ and for $\delta$ clause
(c) of $\circledast_1$ holds.
\end{enumerate}
\end{claim}

\begin{proof}
By \cite{Sh:506}.
\end{proof}

\begin{remark}
1) Connected to $\check I[\lambda]$, see \cite{Sh:506}.
\end{remark}

\begin{claim}
\label{1.10}
Suppose $J$ a $\sigma$-complete ideal on $\delta^*,\mu > \kappa =
\cf(\mu),\mu  = \tlim_J \langle \lambda_{i}:i < \delta\rangle,
\delta^* < \mu,\lambda_{i} = \cf(\lambda_{i})> \delta^*$ for $i <
\delta^*$ and $\lambda = \tcf(\prod\limits_{i < \delta^*} \lambda_i/J)$,
and $\langle f_{\alpha}:\alpha <\lambda\rangle$ exemplifies this.

\Then \, we have
\mn
\begin{enumerate}
\item[$(*)$]    if $\langle u_{\beta}:\beta < \lambda \rangle $  is
a sequence of pairwise disjoint non-empty subsets of  $\lambda $,
each of cardinality  $\le \sigma $ (not $<\sigma$!) and
$\alpha^\ast < \mu^+$,  then we can find $B \subseteq  \lambda$ such that:
\sn
\begin{enumerate}
\item[$(a)$]  $\otp(B) = \alpha^*$,
\sn
\item[$(b)$]   if $\beta \in B,\gamma \in B$ and $\beta < \gamma$
then $\sup(u_\beta) < \min(u_\gamma)$,
\sn
\item[$(c)$]  we can find $s_{\zeta} \in J$  for $\zeta \in
\bigcup\limits_{i \in B} u_{i}$ such that: if $\zeta \in
\bigcup\limits_{\beta \in B}
u_{\beta},\xi \in \bigcup\limits_{\beta \in B} u_{\beta},\zeta < \xi$
and $i \in  \delta \backslash (s_\zeta \cup s_\xi)$, then
$f_{\zeta} (i) < f_{\xi}(i)$.
\end{enumerate}
\end{enumerate}
\end{claim}

\begin{PROOF}{\ref{1.10}}
First assume $\alpha^* < \mu$.
For each regular $\theta < \mu$, as $\theta^+ < \lambda =
\cf(\lambda)$ there is a stationary
$S_{\theta} \subseteq \{\delta < \lambda:\cf(\delta)
= \theta < \delta\}$ which is in $\check I[\lambda]$
(see \cite[1.5]{Sh:420}) which is equivalent (see \cite[1.2(1)]{Sh:420}) to:
\mn
\begin{enumerate}
\item[$(*)$]    there is  $\bar C^\theta = \langle
C^\theta_{\alpha}:\alpha < \lambda \rangle$
\sn
\begin{enumerate}
\item[$(\alpha)$]    $C^\theta_{\alpha}$ a subset of  $\alpha $,
with no accumulation points (in $C^\theta_{\alpha}$),
\sn
\item[$(\beta)$]   $[\alpha \in \nacc(C^\theta_{\beta})
\Rightarrow  C^\theta_{\alpha} = C^\theta_{\beta} \cap \alpha]$,
\sn
\item[$(\gamma)$]    for some club $E^0_{\theta}$ of $\lambda$,
\[
[\delta \in S_{\theta} \cap E^0_{\theta}  \Rightarrow \cf(\delta)=
\theta < \delta \land \delta = \sup(C^\theta_{\delta}) \land \otp
(C^\theta_{\delta}) = \theta].
\]
\end{enumerate}
\end{enumerate}
\mn
Without loss of generality  $S_{\theta}  \subseteq  E^0_{\theta} $,
and $\bigwedge\limits_{\alpha < \delta} \otp(C^\theta_\alpha)
\le  \theta$.  By \cite[2.3,Def.1.3]{Sh:365} for some club
$E_{\theta}$ of  $\lambda,\langle g \ell(C^\theta_{\alpha},
E_{\theta}):\alpha \in S_{\theta} \rangle$
guess clubs (i.e., for every club $E \subseteq E_{\theta}$ of $\lambda $,
for stationarily many  $\zeta \in S_{\theta}$,
$g \ell(C^\theta_{\zeta},E_{\theta}) \subseteq E)$  (remember
$g \ell(C^\theta_{\delta},E_{\theta}) = \{\sup(\gamma \cap
E_{\theta}):\gamma \in C^\theta_{\delta};\gamma > \Min
(E_{\theta})\})$.  Let  $C^{\theta,*}_{\alpha}
= \{\gamma \in C^\theta_{\alpha}:\gamma = \Min
(C^\theta_{\alpha} \backslash \sup (\gamma \cap  E_\theta))\}$,
they have all the properties of the  $C^\theta_{\alpha}$'s and guess
clubs in a weak sense: for every club $E$  of  $\lambda $  for some
$\alpha \in  S_{\theta}  \cap  E$,  if  $\gamma_{1} < \gamma_{2}$ are
successive members of $E$ then $|(\gamma_{1},\gamma_{2}] \cap
C^{\theta,*}_{\alpha}| \le 1$; moreover, the function
$\gamma \mapsto  \sup(E \cap \gamma)$ is one to one on
$C^{\theta,*}_\alpha$.

Now we define by induction on  $\zeta < \lambda $,  an ordinal
$\alpha_{\zeta} $ and functions  $g^\zeta_{\theta}  \in
\prod\limits_{i < \delta^*} \lambda_{i}$ (for each $\theta \in {\Theta} =:
\{\theta:\theta < \mu,\theta$ regular uncountable$\})$.

For given  $\zeta $,  let  $\alpha_{\zeta}  < \lambda $  be minimal
such that:

\[
\xi < \zeta \Rightarrow  \alpha_{\xi} < \alpha_{\zeta}
\]

\[
\xi  < \zeta \; \;   \and\;  \theta \in \Theta \Rightarrow g^\xi_{\theta} <
f_{\alpha_\zeta} \mod J.
\]

\mn
Now  $\alpha_{\zeta} $ exists as  $\langle f_{\alpha}:\alpha < \lambda
\rangle$  is $<_{J}$-increasing cofinal in
$\prod\limits_{i < \delta^*} \lambda_i/J$.  Now for each $\theta \in
{\Theta }$  we define $g^\zeta_{\theta}$ as follows:
\mn
\begin{enumerate}
\item[${{}}$]    for  $i < \delta^*,g^\zeta_{\theta} (i)$  is
$\sup[\{g^\xi_{\theta}(i) + 1:\xi \in C^\theta_{\zeta}\} \cup
\{f_{\alpha_{\zeta}}(i) + 1\}]$ if this number is $< \lambda_{i}$, and
$f_{\alpha_{\zeta}}(i) +1$ otherwise.
\end{enumerate}
\mn
Having made the definition we prove the assertion.
We are given $\langle u_{\beta}:\beta < \lambda \rangle$,
a sequence of pairwise disjoint non-empty subsets of  $\lambda$, each
of cardinality $\le \sigma $  and $\alpha^\ast  < \mu $.
We should find $B$ as promised; let $\theta =:(|\alpha^*| +
|\delta^*|)^+$ so $\theta < \mu$ is regular $> |\delta^*|$.
Let  $E = \{\delta \in E_{\theta}:(\forall \zeta)[\zeta <
\delta \Leftrightarrow  \sup(u_{\zeta}) < \delta \Leftrightarrow
u_{\zeta}  \subseteq  \delta\Leftrightarrow\alpha_\zeta <\delta]\}$.
Choose  $\alpha \in  S_{\theta}  \cap \acc(E)$ such that
$g \ell(C_\zeta^\theta,E_\theta)\subseteq E$;  hence letting
$C_{\alpha}^{\theta,*} = \{\gamma_{i}:i < \theta\}$
(increasing), $\gamma(i) = \gamma_i$, we know that $i < \delta^* \Rightarrow
(\gamma_{i},\gamma_{i+1}) \cap  E \ne \emptyset$.  Now let
$B =: \{\gamma_{5i+3}:i < \alpha^*\}$ we shall prove that $B$ is as required.
For $\alpha\in u_{\gamma(5\zeta + 3)},\zeta<\alpha^*$, let
$s_\alpha^o = \{i < \delta^*:g_\theta^{\gamma(5\zeta + 1)}(i) < f_\alpha(i) <
g^{\gamma(5\zeta + 4)}_\theta(i)\}$, for each $\zeta < \alpha^*$ let
$\langle \alpha_{\zeta,\varepsilon}:\varepsilon <
|u_{\gamma(5\zeta +3)}| \rangle$ enumerate $u_{\gamma(5\zeta + 3)}$
and let

\begin{equation*}
\begin{array}{clcr}
s^1_{\alpha_{\zeta,\varepsilon}} = \{i:\text{ for every }
\xi < \epsilon, f_{\alpha_{\zeta,\xi}}(i)< f_{\alpha_{\zeta,\epsilon}}(i)
&\Leftrightarrow \alpha_{\zeta,\xi} < \alpha_{\zeta,\epsilon} \\
  &\Leftrightarrow f_{\alpha_{\zeta,\xi}}(i)
\le f_{\alpha_{\zeta,\epsilon}}(i)\}.
\end{array}
\end{equation*}

\mn
Lastly, for $\alpha \in \bigcup\limits_{\zeta<\alpha^*} u_{5\zeta + 3}$ let
$s_\alpha=s_\alpha^o \cup s^1_\alpha$ and it is enough to check that
$\langle \zeta_\alpha:\alpha \in B\rangle$ witness that $B$ is as
required.  Also we have to consider $\alpha^* \in [\mu,\mu^+)$, we
prove this by induction on $\alpha^*$ and in the induction step we use
$\theta = (\text{cf}(\alpha^*) + |\delta^*|)^+$ using a similar proof.
\end{PROOF}

\begin{remark}
\label{1.11}
In \ref{1.10}:

\noindent
1) We can avoid guessing clubs.

\noindent
2)  Assume  $\sigma < \theta_{1} < \theta_{2} < \mu $  are regular
and there is $S \subseteq\{\delta <\lambda:\cf(\delta)=
\theta_{1}\}$  from  $I[\lambda]$  such that for every  $\zeta <
\lambda$ (or at least a club) of cofinality  $\theta_{2}$, $S \cap
\zeta$  is stationary and $\langle f_\alpha:\alpha<\lambda\rangle$
obey suitable $\bar C^\theta$ (see \cite[\S2]{Sh:345a}).
Then for some  $A \subseteq  \lambda$  unbounded, for every
$\langle u_{\beta}:\beta < \theta_{2}\rangle$ sequence of pairwise
disjoint non-empty subsets of $A$,  each of cardinality  $< \sigma$
with  $[\min u_{\beta},\sup u_{\beta}]$  pairwise disjoint we have:
for every  $B_{0} \subseteq  A$  of order
type  $\theta_{2}$,  for some  $B \subseteq B_{0}$, $|B| = \theta_{1}$,
(c) of $(*)$ of \ref{1.10} holds.

\noindent
3)  In $(*)$ of \ref{1.10},  $``\alpha^* < \mu"$  can be replaced by
``$\alpha^* < \mu^+$'' (prove by induction on $\alpha^*$).
\end{remark}

\begin{observation}
\label{1.12}
Assume $\lambda < \lambda^{<\lambda},\mu = \Min\{\tau:2^\tau  > \lambda\}$.
\Then \, there are  $\delta,\chi$ and $\cT$,
satisfying the condition $(*)$ below for  $\chi  = 2^\mu $ or at
least arbitrarily large regular $\chi < 2^\mu$
\mn
\begin{enumerate}
\item[$(*)$]    $\cT$ a tree with  $\delta$ levels,
(where $\delta \le \mu)$  with a set  $X$ of $\ge \chi$ \,
$\delta$-branches, and for  $\alpha < \delta,\bigcup\limits_{\beta < \alpha}
|{\cT}_{\beta}| < \lambda$.
\end{enumerate}
\end{observation}

\begin{PROOF}{\ref{1.12}}
So let $\chi \le 2^\mu$ be regular, $\chi > \lambda$.
\bigskip

\noindent
\underline{Case 1}:  $\bigwedge\limits_{\alpha < \mu} 2^{|\alpha|} < \lambda$.
Then ${\cT} = {}^{\mu >}2,{\cT}_{\alpha} = {}^\alpha 2$ are O.K.
(the set of branches ${}^\mu 2$  has cardinality  $2^\mu$).
\bigskip

\noindent
\underline{Case 2}:  Not Case 1.  So for some  $\theta < \mu $,  $2^\theta \ge
\lambda $,  but by the choice of  $\mu,2^\theta \le\lambda $,
so  $2^\theta = \lambda,\theta < \mu$ and so $\theta \le \alpha < \mu
\Rightarrow  2^{|\alpha|} = 2^\theta$.
Note $|{}^{\mu >}2| = \lambda$ as $\mu \le \lambda$.  Note also that
$\mu = \cf(\mu)$ in this case (by the Bukovsky-Hechler theorem).
\bigskip

\noindent
\underline{Subcase 2A}:  $\cf(\lambda)\ne \mu = \cf(\mu)$.

Let ${}^{\mu >}2 = \bigcup\limits_{j < \lambda} B_{j},B_{j}$
increasing with $j,|B_{j}| < \lambda$.  For each $\eta \in
{}^\mu 2$, (as cf$(\lambda) \ne \cf(\mu)$)  for some $j_{\eta} < \lambda$,

\[
\mu  = \sup \{\zeta < \mu:\eta \upharpoonright \zeta \in B_{j_\eta}\}.
\]

\mn
So as $\cf(\chi) \ne \mu$,  for some ordinal $j^* < \lambda$
we have

\[
\{\eta \in {}^\mu 2:j_\eta \le j^*\} \text{ has cardinality } \ge\chi.
\]

\mn
As $\cf(\lambda) \ne \cf(\mu)$  and $\mu \le \lambda$ (by its
definition) clearly  $\mu < \lambda$,  hence $|B_{j^*}| \times \mu <
\lambda$.

Let

\[
{\cT} = \{\eta \upharpoonright \epsilon:\epsilon
< \ell g(\eta) \text{  and } \eta \in  B_{j^*}\}.
\]

\mn
It is as required.
\bigskip

\noindent
\underline{Subcase 2B}:  Not 2A so $\cf(\lambda)= \mu = \cf(\mu)$.

If $\lambda = \mu$  we get $\lambda = \lambda^{<\lambda}$
contradicting an assumption.

So  $\lambda > \mu$, so $\lambda$ singular.  Now if
$\alpha < \mu,\mu  < \sigma_{i} = \cf(\sigma_{i})< \lambda$
for  $i < \alpha $  then (see \cite[1.3(10)]{Sh:345a})
$\max \pcf\{\sigma_{i}:i < \alpha\} \le \prod\limits_{i < \alpha}
\sigma_{i} \le \lambda^{|\alpha|} \le (2^\theta)^{|\alpha|} \le
2^{<\mu} = \lambda$,  but as  $\lambda $ is singular and
$\max \pcf\{\sigma_{i}:i < \alpha\}$  is regular
(see \cite[1.9]{Sh:345a}),  clearly the inequality is strict, i.e.,
$\max \pcf\{\sigma_{i}:i < \alpha\} < \lambda$.  So let  $\langle
\sigma_{i}:i < \mu \rangle$ be a strictly increasing sequence of regulars in
$(\mu,\lambda)$ with limit $\lambda$,  and by \cite[3.4]{Sh:355} there
is ${\cT} \subseteq \prod\limits_{i < \mu} \sigma_{i}$ satisfying
$|\{\nu \upharpoonright i:\nu \in {\cT}\}| \le
\max \pcf\{\sigma_{j}:j < i\} < \lambda$, and number of $\mu$-branches
$> \lambda$.
In fact we can get any regular cardinal in $(\lambda,\pp^+(\lambda))$
in the same way.

Let $\lambda^* = \min\{\lambda':\mu <\lambda'\le \lambda,\cf(\lambda')
=\mu$ and $\pp(\lambda')>\lambda\}$, so (by
\cite[2.3]{Sh:355}),  also $\lambda^\ast$ has those properties and
$\pp(\lambda^*) \ge \pp(\lambda)$.  So if $\pp^+(\lambda^*) =
(2^\mu)^+$ or $\pp(\lambda^*) = 2^\mu$ is singular, we are done.
So assume this fails.

If  $\mu  > \aleph_{0}$, then (as in \cite[3.4]{Sh:430}) $\alpha < 2^\mu
\Rightarrow \cov(\alpha,\mu^+,\mu^+,\mu) < 2^\mu $ and we
can finish as in subcase 2A (actually $\cov(2^{<\mu},\mu^+,\mu^+,\mu)
< 2^\mu$ suffices which holds by the
previous sentence and \cite[5.4]{Sh:355}).
If $\mu  = \aleph_{0}$ all is easy.
\end{PROOF}

\begin{claim}
\label{6.4}
Assume ${\gb}_0 \subseteq \ldots \subseteq {\gb}_k \subseteq
{\gb}_{k+1} \subseteq \cdots$  for  $k < \omega,{\ga} =
\bigcup\limits_{k < \omega} {\gb}_k$  (and $|{\ga}|^+ <
\Min({\ga}))$ and $\lambda \in \pcf({\ga}) \backslash
\bigcup\limits_{k < \omega} \,\pcf({\gb}_k)$.

\noindent
1) We can find finite ${\gd}_k \subseteq \pcf ({\gb}_k
\backslash {\gb}_{k-1})$  (stipulating  ${\gb}_{-1} = \emptyset )$
such that $\lambda \in \pcf(\cup\{{\gd}_k:k < \omega\})$.

\noindent
2) Moreover, we can demand ${\gd}_k \subseteq
\pcf({\gb}_k) \backslash (\pcf({\gb}_{k-1}))$.
\end{claim}

\begin{proof}
We start to repeat the proof of \cite[1.5]{Sh:371} for
$\kappa = \omega$.
But there we apply \cite[1.4]{Sh:371} to $\langle {\gb}_\zeta:
\zeta < \kappa \rangle$ and get $\langle \langle
{\gc}_{\zeta,\ell}:\ell \le n(\zeta)\rangle:\zeta < \kappa\rangle$ and
let $\lambda_{\zeta,\ell} = \max \pcf({\gc}_{\zeta,\ell})$.
Here we apply the same claim (\cite[1.4]{Sh:371}) to  $\langle
{\gb}_{k} \backslash {\gb}_{k-1}:k < \omega\rangle$ to get part (1).
As for part (2), in the proof of \cite[1.5]{Sh:371} we let $\delta =
|{\ga}|^+ + \aleph_{2}$ choose $\langle N_{i}:i<\delta\rangle$, but now we
have to adapt the proof of \cite[1.4]{Sh:371} (applied to ${\ga}$,
$\langle {\gb}_{k}:k < \omega \rangle$, $\langle N_{i}:i < \delta
\rangle)$;  we have gotten there, toward the end, $\alpha < \delta$
such that $E_{\alpha} \subseteq E$. Let  $E_{\alpha} =
\{i_k:k <  \omega\},i_k < i_{k+1}$.  But now instead of applying
\cite[1.3]{Sh:371} to each ${\gb}_{\ell}$ separately,  we try to choose
$\langle {\gc}_{\zeta ,\ell}:\ell \le n(\zeta)\rangle$ by induction
on  $\zeta < \omega$.
For  $\zeta = 0$  we apply \cite[1.3]{Sh371}.  For  $\zeta > 0$, we
apply \cite[1.3]{Sh:371} to ${\gb}_{\zeta}$ but there defining by
induction on $\ell,{\gc}_{\ell} = {\gc}_{\zeta,\ell} \subseteq
{\ga}$ such that $\max(\pcf({\ga} \backslash {\gc}_{\zeta,0}
\backslash \cdots \backslash {\gc}_{\zeta,\ell-1})\cap
\pcf({\frak b}_{\zeta})$ is strictly decreasing with $\ell$.
\end{proof}

\noindent
We use:
\begin{observation}
\label{1.21}
If $|{\ga}_i| < \Min ({\ga}_i)$ for $i<i^*$, then ${\gc} =
\bigcap\limits_{i< i^*} \pcf({\ga}_i)$ has a last element or is empty.
\end{observation}

\begin{PROOF}{\ref{1.21}}
By renaming \wilog \, $\langle |{\ga}_i|:i< i^* \rangle$ is
non-decreasing.  By \cite[1.12]{Sh:345b}
\mn
\begin{enumerate}
\item[$(*)_1$]   ${\gd} \subseteq {\gc}\; \,  \and \;  |{\gd}| < \Min({\gd})
\Rightarrow \pcf({\gd})\subseteq {\gc}$.
\end{enumerate}
\mn
By \cite[2.6]{Sh:371} or \ref{6.7C.1}(2)
\mn
\begin{enumerate}
\item[$(*)_2$]   if $\lambda \in \pcf({\gd}),{\gd}
\subseteq {\gc},|{\gd}| < \Min({\gd})$
then for some ${\ge} \subseteq {\gd}$ we have
$|{\ge}| \le \Min({\ga}_0),\lambda \in \pcf({\ge})$.
\end{enumerate}
\mn
Now choose by induction on $\zeta < |{\ga}_0|^+,\theta_\zeta \in
{\gc}$, satisfying $\theta_\zeta > \max \pcf
\{\theta_\epsilon:\epsilon<\zeta\}$.  If we are stuck in $\zeta$,
$\max \pcf\{\theta_\epsilon:\epsilon<\zeta\}$ is the desired maximum by
$(*)_1$.
If we succeed the cardinal $\theta = \max \pcf\{\theta_\epsilon:
\epsilon < |{\ga}_0|^+\}$ is in $\pcf\{\theta_\epsilon:
\epsilon<\zeta\}$ for some $\zeta < |{\ga}_0|^+$ by
$(*)_2$; easy contradiction.
\end{PROOF}

\begin{conclusion}
\label{1.22}
Assume $\aleph_{0} = \cf(\mu)\le \kappa \le \mu_{0} < \mu,
[\mu' \in (\mu_{0},\mu)\; \,  \and \;  \cf(\mu') \le \kappa
\Rightarrow \pp_{\kappa}(\mu') <\lambda]$ and $\pp^+_{\kappa}(\mu)
> \lambda = \cf(\lambda) > \mu$.
\Then \, we can find $\lambda_{n}$ for $n < \omega,\mu_{0} <
\lambda_{n} < \lambda_{n+1} < \mu $,  $\mu  = \bigcup\limits_{n < \omega}
\lambda_{n}$ and $\lambda = \tcf(\prod\limits_{n<\omega}
\lambda_{n}/J)$  for some ideal $J$ on $\omega $
(extending  $J^{\bd}_{\omega}$).
\end{conclusion}

\begin{PROOF}{\ref{1.22}}
Let ${\ga}\subseteq (\mu_0,\mu)\cap \Reg,
|{\ga}| \le \kappa,\lambda \in \pcf({\ga})$.
Without loss of generality  $\lambda = \max
\pcf({\ga})$,  let $\mu  = \bigcup\limits_{n < \omega} \mu^0_{n}$,
$\mu_{0} \le \mu^0_{n} < \mu^0_{n+1} < \mu$,
let  $\mu ^1_{n} = \mu^0_n + \sup\{\pp_{\kappa}(\mu'):\mu_{0} <
\mu' \le \mu^0_{n}$ and $\cf(\mu') \le \kappa\}$,  by \cite[2.3]{Sh:355}
$\mu^1_{n} < \mu,\mu^1_{n} = \mu^0_n + \sup\{\pp_{\kappa}(\mu'):\mu_{0}
< \mu' < \mu ^1_{n}$ and $\cf(\mu') \le \kappa\}$ and obviously
$\mu^1_{n} \le \mu ^1_{n+1}$;  by replacing by a
subsequence without loss of generality  $\mu^1_{n} < \mu^1_{n+1}$.
Now let ${\gb}_{n} = {\ga} \cap  \mu^1_{n}$ and apply the previous
claim \ref{6.4}:  to ${\gb}_{k} =: {\ga} \cap (\mu^1_{n})^+$,  note:

\[
\max \pcf({\gb}_k) \le \mu^1_{k} < \Min ({\gb}_{k+1}\backslash {\gb}_{k}).
\]
\end{PROOF}

\begin{claim}
\label{1.23}
1) Assume $\aleph_{0} < \cf(\mu) = \kappa < \mu_{0} < \mu,2^\kappa < \mu$
and $[\mu_{0} \le \mu' < \mu \; \,  \and \;  \cf(\mu')\le
\kappa \Rightarrow \pp_\kappa(\mu') < \mu]$.
If  $\mu  < \lambda = \cf(\lambda) < \pp^+(\mu)$
\then \, there is a tree ${\cT}$ with $\kappa$ levels, each level
of cardinality $<\mu,{\cT}$  has exactly $\lambda \,\kappa$-branches.

\noindent
2) Suppose  $\langle \lambda_{i}:i < \kappa \rangle$  is
a strictly increasing sequence of regular cardinals,  $2^\kappa <
\lambda_{0},{\ga} =: \{\lambda_{i}:i < \kappa\},
\lambda = \max \pcf({\ga}),\lambda_{j} > \max \pcf
\{\lambda_{i}:i < j\}$  for each  $j < \kappa$  (or at least
$\sum\limits_{i < j}\lambda_{i} > \max \pcf \{\lambda_{i}:i < j\})$
and ${\ga} \notin J$  where $J = \{{\gb} \subseteq {\ga}:
{\gb}$  is the union of countably many members of
$J_{<\lambda}[{\ga}]\}$  (so $J \supseteq J^{\bd}_{\ga}$
and $cf(\kappa) > \aleph_{0})$.  \Then \, the
conclusion of (1) holds with $\mu = \sum\limits_{i<\kappa }\lambda_{i}$.
\end{claim}

\begin{PROOF}{\ref{1.23}}
1) By (2) and \cite[\S1]{Sh:371} (or can use the conclusion of
\cite[AG,5.7]{Sh:g}).

\noindent
2)  For each ${\gb} \subseteq {\ga}$ define the function
$g_{\gb}:\kappa \rightarrow \Reg$ by

\[
g_{\gb}(i) = \max \pcf[{\gb} \cap \{\lambda_{j}:j < i\}].
\]

\mn
Clearly $[{\gb}_{1}\subseteq {\gb}_{2} \Rightarrow g_{{\gb}_{1}}
\le g_{\gb_2}]$.
As $\cf(\kappa ) > \aleph_{0},J$ is $\aleph_{1}$-complete,  there is
${\gb} \subseteq {\ga},{\gb} \notin  J$  such that:

\[
{\gc} \subseteq {\gb} \;  \and \;  {\gc} \notin J
\Rightarrow \neg g_{\gc} <_{J} g_{\gb}.
\]

\mn
Let  $\lambda^*_{i} = \max \pcf({\gb} \cap \{\lambda_{j}:j < i\})$.
For each  $i$  let  ${\gb}_i = {\gb} \cap
\{\lambda_{j}:j < i\}$ and $\langle \langle f^{\gb}_{\lambda,\alpha}:
\alpha < \lambda \rangle:\lambda \in \pcf({\gb})\rangle$ be
as in \cite[\S1]{Sh:371}.

Let

\[
{\cT}^0_i = \{ \underset{0<\ell<n} \Max
f^{\gb}_{\lambda_\ell,\alpha_\ell} \upharpoonright {\gb}_i:
\lambda_\ell \in \pcf({\gb}_i),\alpha_\ell < \lambda_\ell,n < \omega\}.
\]

\mn
Let ${\cT}_{i} = \{f \in {\cT}^0_{i}:$ for every $j < i,f \upharpoonright
{\gb}_{j} \in {\cT}^0_{j}$ moreover for some $f' \in
\prod\limits_{j<\kappa}\lambda_j$, for every $j$, $f' \upharpoonright
{\gb}_j \in {\cT}_j^0$ and $f\subseteq f'\}$,  and ${\cT} =
\bigcup\limits_{{i} < {\kappa}} {\cT}_{i}$,  clearly it is a tree,
${\cT}_{i}$ its $i$th level (or empty),
$|{\cT}_{i}| \le \lambda^*_{i}$.  By \cite[1.3,1.4]{Sh:371} for every
$g \in \prod {\gb}$  for some  $f \in \prod {\gb},
\bigwedge\limits_{i < \kappa} f \upharpoonright {\gb}_{i} \in
{\cT}^0_i$ hence $\bigwedge\limits_{i < \kappa} f \upharpoonright
{\gb}_{i} \in {\cT}_{i}$.
So  $|{\cT}_{i}| = \lambda^*_{i}$,  and ${\cT}$ has $\ge
\lambda \,\kappa$-branches.  By the observation
below we can finish (apply it essentially to ${\cF} = \{\eta$: for some
$f\in \prod {\gb}$ for $i<\kappa$ we have $\eta(i)= f
\upharpoonright {\gb}_i$ and for every $i < \kappa,f
\upharpoonright {\gb}_{i} \in {\cT}^0_{i}\})$,  then
find $A \subseteq  \kappa,\kappa \setminus A\in J$ and $g^*\in
\prod\limits_{i< \kappa}(\lambda_i +1)$ such that $Y' =: \{f \in F:
f \upharpoonright A <  g^* \upharpoonright A\}$  has cardinality
$\lambda$ and then the tree will be ${\cT}'$ where ${\cT}'_{i} =:
\{f \upharpoonright {\gb}_{i}:f \in  Y'\}$
and ${\cT}' = \bigcup\limits_{i < \kappa}{\cT}'_{i}$.
(So actually this proves that if we have such a tree with
$\ge \theta (\cf(\theta) > 2^\kappa$) \,
$\kappa$-branches then there is one with exactly
$\theta$ \, $\kappa$-branches.)
\end{PROOF}

\begin{observation}
\label{1.24}
If ${\cF} \subseteq \prod\limits_{i < \kappa} \lambda_{i}$,  $J$
an $\aleph_{1}$-complete ideal on  $\kappa $,  and
$[f \neq  g \in  {\cF} \Rightarrow  f \neq_{J} g]$  and $|{\cF}|
\ge \theta,\cf(\theta) > 2^\kappa $, then for some $g^*
\in \prod\limits_{{i} < {\kappa}} (\lambda_{i} + 1)$  we have:
\mn
\begin{enumerate}
\item[$(a)$]    $Y = \{f \in {\cF}:f<_{J} g^*\}$ has cardinality $\theta$,
\sn
\item[$(b)$]    for $f'<_{J} g^*$, we have
$|\{f\in {\cF}:f \le_{J}f'\}|< \theta$,
\sn
\item[$(c)$]  there
\footnote{Or straightening clause (i) see the proof of \ref{1.25}}
are  $f_{\alpha} \in Y$ for $\alpha < \theta$ such that:
$f_{\alpha} <_{J} g^*,[\alpha < \beta < \theta \Rightarrow
\neg f_{\beta} <_{J} f_{\alpha}]$.
\end{enumerate}
\mn
(Also in \cite[\S1]{Sh:829}).
\end{observation}

\begin{PROOF}{\ref{1.24}}
Let $Z =: \{g: g \in \prod\limits_{{i} < {\kappa}}
(\lambda_{i} + 1)$  and $Y_{g} =: \{f \in {\cF}:f \le_{J} g\}$ has
cardinality  $\ge  \theta\}$.
Clearly $\langle\lambda_{i}:i <\kappa\rangle\in Z$ so there is
$g^* \in Z$ such that:  $[g' \in Z \Rightarrow  \neg g' <_{J} g^*]$;
so clause (b) holds.
Let  $Y = \{f \in {\cF}:f <_{J} g^*\}$,  easily $Y\subseteq
 Y_{g^*}$ and $|Y_{g^*} \setminus Y|\le 2^\kappa$ hence $|Y| \ge
\theta$,  also clearly $[f_{1} \ne f_{2} \in {\cF} \;\,  \and \;  f_{1}
\le_J f_{2} \Rightarrow  f_{1} <_{J} f_{2}]$.  If
(a) fails, necessarily by the previous sentence $|Y| > \theta$.
For each  $f \in  Y$  let  $Y_{f} = \{h \in Y:h \le_J f\}$,  so by
clause (b) we have
$|Y_{f}| < \theta $  hence by the Hajnal free subset theorem for some
$Z' \subseteq Z$,  $|Z'| = \lambda ^+$,  and $f_{1} \ne  f_{2} \in  Z'
\Rightarrow  f_{1} \notin Y_{f_{2}}$ so  $[f_{1} \ne  f_{2} \in  Z'
\Rightarrow  \neg f_{1} <_{J} f_{2}]$.
But there is no such  $Z'$  of cardinality  $> 2^\kappa$
(\cite[2.2,p.264]{Sh:111}) so clause (a) holds.
As for clause (c): choose  $f_{\alpha} \in {\cF}$ by induction on  $\alpha$,
such that  $f_{\alpha}  \in  Y \setminus \bigcup_{{\beta} <
{\alpha}} Y_{f_{\beta}}$;  it exists by cardinality considerations and
$\langle f_{\alpha}:\alpha < \theta \rangle $  is as required (in (c)).
\end{PROOF}

\begin{observation}
\label{1.25}
Let $\kappa < \lambda$ be regular uncountable,
$2^\kappa < \mu_{i} < \lambda $ (for $i<\kappa$),
$\mu_{i}$ increasing in  $i$.  The following are equivalent:
\mn
\begin{enumerate}
\item[$(A)$]    there is ${\cF} \subseteq {}^\kappa \lambda$ such that:
\sn
\begin{enumerate}
\item[$(i)$]   $|{\cF}| = \lambda$,
\sn
\item[$(ii)$]   $|\{f \upharpoonright i:f \in {\cF}\}| \le \mu_{i}$,
\sn
\item[$(iii)$]  $[f \ne  g \in {\cF} \Rightarrow f \ne_{J^{\bd}_{\kappa}} g]$;
\end{enumerate}
\item[$(B)$]    there be a sequence  $\langle \lambda_{i}:i < \kappa
\rangle$ such that:
\sn
\begin{enumerate}
\item[$(i)$]  $2^\kappa < \lambda_{i} = \text{ cf}(\lambda_{i}) \le \mu_{i}$,
\sn
\item[$(ii)$]   $\max \pcf\{\lambda_{i}:i<\kappa\}=\lambda$,
\sn
\item[$(iii)$]    for $j < \kappa,\mu_{j} \ge \max \pcf
\{\lambda_{i}:i < j\}$;
\end{enumerate}
\item[$(C)$]    there is an increasing sequence $\langle {\ga}_{i}:
i < \kappa \rangle$  such that $\lambda \in \pcf
(\bigcup\limits_{{i} < {\kappa}} {\ga}_{i}),\pcf({\ga}_{i}) \subseteq
\mu_{i}$  (so $\Min(\bigcup\limits_{{i} < {\kappa}} {\ga}_{i}) >
|\bigcup\limits_{{i} < {\kappa}} {\ga}_{i}|)$.
\end{enumerate}
\end{observation}

\begin{PROOF}{\ref{1.25}}

\noindent
$(B) \Rightarrow (A)$:  By \cite[3.4]{Sh:355}.

\noindent
$(A) \Rightarrow (B)$:  If $(\forall \theta )[\theta \ge
2^\kappa \Rightarrow  \theta^\kappa \le \theta^+]$ we can directly
prove (B) if for a club of $i<\kappa,\mu_i > \bigcup\limits_{j<i}\mu_j$, and
contradict (A) if this fails.   Otherwise every normal filter
$D$ on  $\kappa$  is nice (see \cite[\S1]{Sh:386}).
Let ${\cF}$ exemplify (A).

Let  $K = \{(D,g):D$  a normal filter on  $\kappa $,
$g \in {}^\kappa(\lambda +1),\lambda = |\{f \in {\cF}:f <_{D} g\}|\}$.
Clearly  $K$  is not empty (let  $g$  be constantly  $\lambda )$  so
by \cite{Sh:386} we can find $(D,g) \in  K$  such that:
\mn
\begin{enumerate}
\item[$(*)_1$]    if  $A \subseteq \kappa,A \ne \emptyset \mod D,g_{1}
<_{D+A} g$  then  $\lambda > |\{f \in {\cF}:f <_{D+A} g_{1}\}|$.
\end{enumerate}
\mn
Let ${\cF}^* =\{f\in {\cF}:f<_{D} g\}$,
so (as in the proof of \ref{1.23}) $|{\cF}^*| = \lambda$.

We claim:
\mn
\begin{enumerate}
\item[$(*)_2$]     if  $h \in {\cF}^*$ then $\{f \in {\cF}^*:
\neg h \le_{D} f\}$  has cardinality  $< \lambda$.
\end{enumerate}
\mn
[Why?  Otherwise for some $h \in {\cF}^*,{\cF}' =:
\{f \in {\cF}^*:\neg h \le_{D} f\}$ has cardinality $\lambda$,
for  $A \subseteq \kappa$ let ${\cF}'_{A} = \{f \in {\cF}^*:
f \upharpoonright A \le h \upharpoonright A\}$ so ${\cF}'
= \bigcup \{{\cF}'_{A}: A \subseteq \kappa,
A \ne  \emptyset \mod D\}$,  hence (recall that $2^\kappa <
\lambda$) for some $A \subseteq \kappa,A \ne \emptyset \mod D$
and $|{\cF}'_{A}| = \lambda$; now $(D + A,h)$  contradicts $(*)_{1}$].

By $(*)_{2}$ we can choose by induction on  $\alpha < \lambda $,
a function $f_{\alpha} \in F^*$ such that
$\bigwedge\limits_{{\beta} < {\alpha}} f_{\beta} <_{D} f_{\alpha} $.
By \cite[1.2A(3)]{Sh:355} $\langle f_{\alpha}:\alpha < \lambda \rangle$
has an e.u.b. $f^*$.
Let $\lambda_{i} = \cf(f^*(i))$, clearly  $\{i < \kappa:
\lambda_{i} \le 2^\kappa\} = \emptyset \mod D$,  so \wilog \,
$\bigwedge\limits_{{i} < {\kappa}} \cf(f^*(i)) > 2^\kappa$
so $\lambda_{i}$ is regular $\in (2^\kappa ,\lambda]$,
and $\lambda = \tcf(\prod\limits_{{i} < {\kappa}} \lambda_{i}/D)$.
Let  $J_{i} = \{A \subseteq i:\max \pcf\{\lambda_{j}:j \in A\}
\le \mu_{i}\}$;  so (remembering (ii) of (A)) we can find
$h_{i} \in \prod\limits_{{j} < {i}}f^*(i)$  such that:
\mn
\begin{enumerate}
\item[$(*)_3$]     if  $\{j:j < i\} \notin J_{i}$,  then for every
$f \in {\cF},f \upharpoonright i <_{J_{i}}h_{i}$.
\end{enumerate}
\mn
Let $h \in \prod\limits_{{i} < {\kappa}} f^*(i)$  be defined by:

\noindent
$h(i) = \sup\{h_{j}(i):j \in (i,\kappa)$ and $\{j:j < i\} \notin  J_{i}\}$.
As $\bigwedge\limits_{i} \cf[f^*(i)] > 2^\kappa $,  clearly
$h < f^*$ hence by the choice of $f^*$ for some
$\alpha(*) < \lambda $  we have:  $h <_{D} f_{\alpha(*)}$ and
let  $A =: \{i < \kappa:h(i) < f_{\alpha(*)}(i)\}$, so  $A \in  D$.
Define  $\lambda'_i$ as follows:  $\lambda '_{i}$  is
$\lambda_{i}$ if  $i \in  A$,  and is $(2^\kappa)^+$ if  $i \in
\kappa \backslash A$.
Now  $\langle \lambda'_{i}:i < \kappa \rangle $  is as required in (B).

\noindent
$(B) \Rightarrow (C)$:  Straightforward.

\noindent
$(C) \Rightarrow (B)$:  By \cite[\S1]{Sh:371}.
\end{PROOF}

\begin{claim}
\label{1.26}
If ${\cF} \subseteq {}^\kappa \Ord,
2^\kappa < \theta = \cf(\theta) \le |{\cF}|$
\then \, we can find $g^* \in {}^\kappa \Ord$ and a
proper ideal $I$ on $\kappa$ and $A \subseteq \kappa,A \in I$ such that:
\mn
\begin{enumerate}
\item[$(a)$]    $\prod\limits_{{i} < {\kappa}} g^*(i)/I$  has true
cofinality  $\theta$,  and for each  $i \in \kappa \setminus A$ we
have $\cf[g^*(i)] > 2^\kappa$,
\sn
\item[$(b)$]    for every  $g \in {}^\kappa \Ord$ satisfying
$g \upharpoonright A = g^* \upharpoonright A$,
$g\upharpoonright (\kappa \backslash A) < g^* \upharpoonright
(\kappa \backslash A)$ we can find $f \in {\cF}$ such that:
$f \upharpoonright A = g^* \upharpoonright A,g \upharpoonright
(\kappa \backslash A) < f\upharpoonright (\kappa \backslash A)
< g^* \upharpoonright (\kappa \backslash A)$.
\end{enumerate}
\end{claim}

\begin{PROOF}{\ref{1.26}}
As in \cite[3.7]{Sh:410}, proof of $(A) \Rightarrow (B)$.
(In short let $f_{\alpha} \in {\cF}$ for $\alpha < \theta$
be distinct,  $\chi$  large enough,
$\langle N_i:i < (2^\kappa)^+\rangle$ as there,
$\delta_{i} =: \sup(\theta \cap N_{i}),g_{i} \in {}^\kappa \Ord,
g_{i}(\zeta ) =: \Min[N \cap \Ord \backslash f_{\delta_{i}}
(\zeta)],A \subseteq \kappa$  and $S \subseteq
\{i < (2^\kappa)^+:\cf(i) = \kappa^+\}$ stationary,
$[i \in  S \Rightarrow  g_{i} = g^*]$,
$[\zeta < \alpha \; \, \and \;  i \in  S \Rightarrow [f_{\delta_{i}}(\zeta)
= g^* (\zeta) \equiv  \zeta \in A]$  and for some
$i(*) < (2^\kappa)^+$,  $g^* \in N_{i(*)}$, so
$[\zeta \in\kappa \setminus A \Rightarrow \cf(g^*(\zeta)) >
    2^\kappa]$.
\end{PROOF}

\begin{claim}
\label{1.27}
Suppose $D$ is a $\sigma$-complete filter on
$\theta = \cf(\theta),\kappa$ an infinite cardinal,
$\theta > |\alpha|^\kappa$ for $\alpha<\sigma$,  and for each $\alpha <
\theta $,  $\bar \beta = \langle \beta^\alpha_{\epsilon}:\epsilon  <
\kappa \rangle $  is a sequence of ordinals.
\Then \, for every $X \subseteq \theta,X \ne \emptyset \mod D$ there is
$\langle \beta^*_{\epsilon}:\epsilon < \kappa \rangle$ (a
sequence of ordinals) and $w \subseteq \kappa$ such that:
\mn
\begin{enumerate}
\item[$(a)$]    $\epsilon \in \kappa \backslash w \Rightarrow \sigma \le
\cf(\beta^*_{\epsilon}) \le \theta$,
\sn
\item[$(b)$]   if  $\beta'_{\epsilon} \le \beta^*_{\epsilon}$ and
$[\epsilon \in w \equiv \beta'_{\epsilon} = \beta^*_{\epsilon}]$, then
$\{\alpha \in X$: for every  $\epsilon  < \kappa $ we have
$\beta'_{\epsilon} \le  \beta^\alpha_{\epsilon} \le
 \beta^*_{\epsilon}$ and $[\epsilon \in w \equiv
\beta^\alpha_{\epsilon}  = \beta^*_{\epsilon}]\} \ne \emptyset \mod D$.
\end{enumerate}
\end{claim}

\begin{PROOF}{\ref{1.27}}
Essentially by the same proof as \ref{1.26}
(replacing  $\delta_{i}$ by $\Min\{\alpha \in  X$: for every
$Y \in  N_{i} \cap  D$ we have $\alpha \in Y\})$.
See more \cite[\S6]{Sh:513}.  (See \cite[\S7]{Sh:620}).
\end{PROOF}

\begin{remark}
\label{1.28}
We can rephrase the conclusion as:
\mn
\begin{enumerate}
\item[$(a)$]   $B =: \{\alpha\in X$: if $\epsilon\in w$ then
$\beta^\alpha_\epsilon = \beta^*_\epsilon$, and: if
$\epsilon\in\kappa \setminus w$ then $\beta^\alpha_\epsilon$ is
$<\beta^*_\epsilon$ but $> \sup\{\beta^*_\zeta:\zeta
<\epsilon,\beta^\alpha_\zeta<\beta^*_\epsilon\}\}$ is $\ne \emptyset \mod D$
\sn
\item[$(b)$]    If $\beta'_\epsilon < \beta^*_\epsilon$ for $\epsilon
\in\kappa \setminus w$ then $\{\alpha\in B$: if $\epsilon\in\kappa
\setminus w$ then $\beta_\epsilon^\alpha >\beta'_\epsilon\} \ne
\emptyset \mod D$
\sn
\item[$(c)$]    $\epsilon\in\kappa \setminus w
\Rightarrow \cf(\beta'_\epsilon)$ is $\le\theta$ but $\ge \sigma$.
\end{enumerate}
\end{remark}

\begin{remark}
\label{1.28a}
If $|{\ga}| < \min({\ga}),{\cF} \subseteq \Pi {\ga}$,
$|{\cF}| = \theta = \cf(\theta) \notin \pcf({\ga})$ and
even $\theta>\sigma = \sup(\theta^+\cap \pcf({\ga}))$ then
for some $g\in \Pi {\ga}$, the set $\{f \in {\cF}:f<g\}$ is unbounded
in $\theta$ (or use a $\sigma$-complete $D$ as in \ref{1.28}).
(This is as $\Pi {\ga} / J_{<\theta}[{\ga}]$ is
$\min(\pcf({\ga}) \setminus \theta)$-directed as the
ideal $J_{<\theta}[{\ga}]$ is generated by $\le\sigma$ sets;
this is discussed in \cite[\S6]{Sh:513}.)
\end{remark}

\begin{remark}
\label{1.29}
It is useful to note that \ref{1.27} is
useful to use \cite[\S4,5.14]{Sh:462}: e.g., for if $n<\omega$,
$\theta_0 < \theta_1 < \cdots < \theta_n$, satisfying $(*)$ below,
for any $\beta'_\epsilon\le \beta^*_\epsilon$ satisfying
$[\epsilon\in w \equiv \beta'_\epsilon < \beta^*_\epsilon]$
we can find $\alpha<\gamma$ in $X$ such that:

\[
\epsilon \in w \equiv \beta^\alpha_\epsilon = \beta^*_\epsilon,
\]

\[
\{\epsilon,\zeta\}\subseteq \kappa \setminus w \;\; \,   \and \; 
\{\cf(\beta^\ast_\varepsilon),\cf(\beta^*_\zeta)\} \subseteq
[\theta_\ell,\theta_{\ell +1})) \; \, \and \;  \ell \text{ even }\Rightarrow
\beta^\alpha_\epsilon < \beta^\gamma_\zeta,
\]

\[
\{\epsilon,\zeta\}\subseteq\kappa \setminus w \;\; \,  \and \; 
\{\cf(\beta^*_\varepsilon),\cf(\beta^*_\zeta)\}\subseteq
[\theta_\ell,\theta_{\ell +1}) \;\; \,  \and  \; \ell \text{ odd }\Rightarrow
\beta_\epsilon^\gamma <\beta_\zeta^\alpha
\]

\mn
where
\mn
\begin{enumerate}
\item[$(*)$]   $(a) \quad \epsilon \in\kappa \setminus w
\Rightarrow \cf(\beta^*_\epsilon)\in [\theta_0,\theta_n)$, and
\sn
\item[${{}}$]   $(b) \quad \max \pcf[\{\cf(\beta^*_\epsilon):
\epsilon\in\kappa \setminus w\}\cap \theta_\ell] \le \theta_\ell$
(which holds if $\theta_\ell = \sigma^+_\ell$, $\sigma^\kappa_\ell
= \sigma_\ell$

\hskip25pt for $\ell \in\{\ell,\ldots,n\}$).\hfill $\square$
\end{enumerate}
\end{remark}
\newpage

\section {Nice generating sequences}

\begin{claim}
\label{1.31}
For any ${\ga}$, $|{\ga}| < \Min({\ga})$, we can find $\bar{\gb} = \langle
{\gb}_{\lambda}:\lambda \in {\ga}\rangle$  such that:
\mn
\begin{enumerate}
\item[$(\alpha)$]    $\bar{\gb}$ is a generating sequence, i.e.
\[
\lambda \in {\ga} \Rightarrow  J_{\le \lambda}[{\ga}] =
J_{<\lambda }[{\ga}] + {\gb}_{\lambda},
\]
\sn
\item[$(\beta)$]    $\bar{\gb}$ is smooth, i.e., for $\theta<\lambda$
in ${\ga}$,
\[
\theta \in {\gb}_{\lambda}\Rightarrow {\gb}_{\theta} \subseteq {\gb}_{\lambda},
\]
\sn
\item[$(\gamma)$]    $\bar{\gb}$ is closed, i.e., for $\lambda \in
{\ga}$ we have ${\gb}_{\lambda} = {\ga} \cap \pcf({\gb}_{\lambda})$.
\end{enumerate}
\end{claim}

\begin{definition}
\label{1.32}
1) For a set $a$ and set ${\ga}$ of regular cardinals
let $\Ch^{\ga}_a$ be the function with domain $a \cap {\ga}$
defined by $\Ch^{\ga}_a(\theta) = \sup(a \cap \theta)$.

\noindent
2) We may write $N$ instead of $|N|$, where $N$ is a model (usually an
   elementary submodel of $({\cH}(\chi),\in,<^*_\chi)$ for some
   reasonable $\chi$.
\end{definition}

\begin{observation}
\label{1.33}
If ${\ga} \subseteq a$ and $|a| < \Min({\ga})$ then $\ch^{\ga}_a
\in \Pi{\ga}$.
\end{observation}

\begin{PROOF}{\ref{1.31}}
Let $\langle {\gb}_{\theta}[{\ga}]:\theta \in
\pcf({\ga})\rangle$ be as in \cite[2.6]{Sh:371} or Definition
\cite[2.12]{Sh:506}.  For $\lambda \in {\ga}$,  let $\bar f^{\ga,\lambda} =
\langle f^{\ga,\lambda}_{\alpha}:\alpha < \lambda\rangle$  be a
$<_{J_{<\lambda}[{\ga}]}$-increasing cofinal sequence of members of
$\prod {\ga}$,  satisfying:
\mn
\begin{enumerate}
\item[$(*)_1$]    if $\delta < \lambda$,$|{\ga}| < \cf(\delta) <
\Min({\ga})$ and $\theta \in {\ga}$ then:
\[
f^{{\ga},\lambda}_{\delta}(\theta) = \Min
\{\bigcup\limits_{{\alpha} \in {C}} f^{{\ga},\lambda}_{\alpha}(\theta):
C \text{ a club of } \delta\}
\]
\mn
[exists by \cite[Def.3.3,(2)$^b$ + Fact 3.4(1)]{Sh:345a}].
\end{enumerate}
\mn
Let  $\chi = \beth_{\omega}(\sup({\ga}))^+$ and $\kappa$ satisfies
$|{\ga}| < \kappa = \cf(\kappa) < \Min({\ga})$ (without loss of
generality there is such $\kappa$) and let $\bar N = \langle N_{i}:i <
\kappa \rangle$  be an increasing continuous sequence of elementary
submodels of $({\cH}(\chi),\in ,<^*_{\chi})$, $N_{i} \cap \kappa$
an ordinal, $\bar N \upharpoonright (i + 1) \in  N_{i+1},
\|N_{i}\| < \kappa $, and ${\ga},\langle \bar f^{{\ga},\lambda}:
\lambda \in {\ga} \rangle$ and $\kappa$ belong to $N_{0}$.  Let
$N_{\kappa}  = \bigcup\limits_{{i} < {\kappa}} N_{i}$.
Clearly by \ref{1.33}
\mn
\begin{enumerate}
\item[$(*)_2$]    $\Ch^{\ga}_{N_i} \in \Pi{\ga}$ for $i \le \kappa$.
\end{enumerate}
\mn
Now for every $\lambda \in {\ga}$ the sequence
$\langle \Ch^{\ga}_{N_i}(\lambda):i \le \kappa \rangle$ is
increasing continuous (note that $\lambda \in N_0 \subseteq N_i
\subseteq N_{i+1}$ and $N_i,\lambda \in N_{i+1}$ hence $\sup(N_i \cap
\lambda) \in N_{i+1} \cap \lambda$ hence $\Ch^{\frak a}_{N_i}(\lambda)$
is $< \sup(N_{i+1} \cap \lambda))$.  Hence $\{\Ch^{\ga}_{N_i}
(\lambda):i < \kappa\}$ is a club of $\Ch^{\ga}_{N_\kappa}
(\lambda)$; moreover, for every
club $E$ of $\kappa$ the set $\{\Ch^{\ga}_{N_i}(\lambda):i
\in E\}$ is a club of $\Ch^{\frak a}_{N_\kappa}(\lambda)$.  Hence by $(*)_1$,
for every $\lambda \in {\ga}$,  for some club $E_{\lambda}$ of $\kappa$,
\mn
\begin{enumerate}
\item[$(*)_3$]    $(\alpha) \quad$ if $\theta \in {\ga}$ and $E
\subseteq E_\lambda$ is a club of $\kappa$ then
$f^{{\ga},\lambda}_{\sup (N_{\kappa} \cap \lambda)}(\theta) =
\bigcup\limits_{\alpha \in E} f^{{\ga},\lambda}_{\sup(N_\alpha\cap\lambda)}
(\theta)$
\sn
\item[${{}}$]   $(\beta) \quad f^{{\ga},\lambda}_{\sup(N_\kappa
\cap \lambda)} (\theta) \in c \ell(\theta \cap N_\kappa)$, (i.e., the
closure as a set of ordinals).
\end{enumerate}
\mn
Let $E = \bigcap\limits_{{\lambda} \in {\ga}}E_{\lambda} $,  so  $E$
is a club of  $\kappa $.  For any  $i < j < \kappa $
let

\[
{\gb}^{i,j}_{\lambda} = \{\theta \in {\ga}:\Ch^{\ga}_{N_i}(\theta)
< f^{{\ga},\lambda}_{\sup (N_{j}\cap \lambda)}(\theta)\}.
\]

\mn
\begin{enumerate}
\item[$(*)_4$]    for  $i < j < \kappa$ and $\lambda \in {\ga}$,  we have:
\sn
\begin{enumerate}
\item[$(\alpha)$]   $J_{\le \lambda}[{\ga}] = J_{<\lambda}
[{\ga}] + {\gb}^{i,j}_{\lambda}$ (hence ${\gb}^{i,j}_{\lambda}
= {\gb}_{\lambda}[\bar{\ga}] \mod J_{<\lambda}[{\ga}])$,
\sn
\item[$(\beta)$]    ${\gb}^{i,j}_{\lambda} \subseteq
\lambda^+ \cap {\ga}$,
\sn
\item[$(\gamma)$]    $\langle {\gb}^{i,j}_{\lambda}:
\lambda \in {\ga}\rangle  \in  N_{j+1}$,
\sn
\item[$(\delta)$]   $f^{{\ga},\lambda}_{\sup(N_{\kappa} \cap
\lambda)} \le \Ch^{\ga}_{N_\kappa} = \langle \sup(N_{\kappa}
\cap \theta):\theta \in {\ga}\rangle$.
\end{enumerate}
\end{enumerate}
\mn
[Why?
\medskip

\noindent
\underline{Clause $(\alpha)$}:  First as $\Ch^{\ga}_{N_i} \in \Pi{\ga}$
(by \ref{1.33}) there is $\gamma < \lambda$ such that
$\Ch^{\ga}_{N_i} <_{J_{=\lambda}[{\ga}]} f^{{\ga},\lambda}_\gamma$
and as ${\ga} \cup \{{\ga},N_i\} \subseteq \Ch^{\ga}_{N_{i+1}}$
clearly $\Ch^{\ga}_{N_i} \in N_{i+1}$
hence \wilog \, $\gamma \in \lambda \cap N_{i+1}$ but $i+1 \le j$
hence $N_{i+1} \subseteq N_j$ hence $\gamma \in N_j$ hence
$\gamma < \sup(N_j \cap \lambda)$ hence $f^{{\ga},\lambda}_\gamma
<_{J_{= \lambda}[{\ga}]} f^{{\ga},\lambda}_{\sup(N_j \cap \lambda)}$.
Together $\Ch^{\frak a}_{N_i} <_{J_{=\lambda}[{\ga}]}
f^{{\ga},\lambda}_{\sup(N_j \cap \lambda)}$
hence by the definition of ${\gb}^{i,j}_\lambda$ we have ${\ga}
\backslash {\gb}^{i,j}_\lambda \in J_{= \lambda}[{\ga}]$ hence
$\lambda \notin \pcf({\frak a} \backslash {\gb}^{i,j}_\lambda)$
so $J_{\le \lambda}[{\ga}] \subseteq J_{< \lambda}[{\ga}] +
{\gb}^{i,j}_\lambda$.

Second, $(\Pi{\ga},<_{J_{\le \lambda}[{\ga}]})$ is
$\lambda^+$-directed hence there is $g \in \Pi{\ga}$ such that
$\alpha < \lambda \Rightarrow f^{{\ga},\lambda}_\alpha <_{J_{\le
\lambda}[{\ga}]} g$. As $\bar f^{{\ga},\lambda} \in N_0$
\wilog \, $g \in N_0$ hence $g \in N_i$ so $g < \Ch^{\ga}_{N_i}$.
By the choice of $g,f^{{\ga},\lambda}_{\sup(N_j \cap
\lambda)} <_{J_{\le \lambda}[{\ga}]} g$ so together
$f^{{\ga},\lambda}_{\sup(N_j \cap \lambda)} <_{J_{\le\lambda}[{\ga}]}
\Ch^{\ga}_{N_i}$ hence ${\gb}^{i,j}_\lambda \in J_{\le \lambda}[{\ga}]$.  As
$J_{< \lambda}[{\ga}] \subseteq J_{\le \lambda}[{\ga}]$
clearly $J_{< \lambda}[{\ga}] + {\gb}^{i,j}_\lambda \subseteq
J_{\le \lambda}[{\ga}]$.  Together we are done.
\medskip

\noindent
\underline{Clause $(\beta)$}:  Because $\Pi({\ga} \backslash
\lambda^+)$ is $\lambda^+$-directed we have $\theta \in {\ga}
\backslash \lambda^+ \Rightarrow \{\theta\} \notin J_{\le \lambda}[{\ga}]$.
\medskip

\noindent
\underline{Clause $(\gamma)$}:  As $\Ch^{\ga}_{N_i},
f^{\lambda,{\ga}}_{\sup(N_j \cap \lambda)},\bar f$ belongs to $N_{j+1}$.
\medskip

\noindent
\underline{Clause $(\delta)$}:  For $\theta \in {\frak a}(\subseteq N_0)$ we
have $f^{{\frak a},\lambda}_{\sup(N_\kappa \cap \lambda)}(\theta) =
\cup\{f^{{\frak a},\lambda}_{\sup(N_\varepsilon \cap
\lambda)}(\theta):\varepsilon \in E_\lambda\} \le \sup(N_\kappa
\cap \theta)$.

So we have proved $(*)_4$.]
\mn
\begin{enumerate}
\item[$(*)_5$]   $\varepsilon(*) < \kappa$ when $\varepsilon(*) =
\cup\{\varepsilon_{\lambda,\theta}:\theta <\lambda$ are from ${\ga}\}$
where $\varepsilon_{\lambda,\theta} = \Min\{\varepsilon <
\kappa$: if $f^{{\ga},\lambda}_{\sup(N_\kappa \cap \lambda)}
(\theta) < \sup(N_\kappa \cap \theta)$ then $f^{{\ga},
\lambda}_{\sup(N_\kappa \cap \lambda)}(\theta) < \sup(N_\varepsilon
\cap \theta)\}$.
\end{enumerate}
\mn
[Why?  Obvious.]
\mn
\begin{enumerate}
\item[$(*)_6$]   $f^{{\ga},\lambda}_{\sup(N_\kappa \cap
\lambda)} \restriction {\gb}^{i,j}_\lambda = \Ch^{\ga}_{N_\kappa} \restriction
{\gb}^{i,j}_\lambda$ when $i <j$ are from $E \backslash \varepsilon(*)$.
\end{enumerate}
\mn
[Why? Let $\theta \in {\gb}^{i,j}_\lambda$, so by $(*)_3(\beta)$
we know that $f^{{\ga},\lambda}_{\sup(N_\kappa \cap
\lambda)}(\theta)\le \Ch^{\ga}_{N_\kappa}(\theta)$.
If the inequality is strict then there is $\beta \in N_\kappa
\cap\theta$ such that $f^{{\ga},\lambda}_{\sup(N_\kappa \cap
\lambda)}(\theta) \le \beta < \Ch^{\ga}_{N_\kappa}(\theta)$
hence for some $\varepsilon < \kappa,\beta \in
N_\varepsilon$ hence $\zeta \in (\varepsilon,\kappa) \Rightarrow
f^{{\ga},\lambda}_{\sup(N_\kappa\cap\lambda)}(\theta) <
\Ch^{\ga}_{N_\zeta}(\theta)$ hence (as ``$i \ge
\varepsilon_{\lambda,\theta}$" holds) we
have $f^{{\ga},\lambda}_{\sup(N_\kappa \cap \lambda)}(\theta) <
\Ch^{\ga}_{N_i}(\theta)$ so $f^{{\ga},\lambda}_{\sup(N_j
\cap \lambda)}(\theta) \le f^{{\ga},\lambda}_{\sup(N_\kappa
\cap \lambda)}(\theta) < \Ch^{\ga} _{N_i(\theta)}$,
(the first inequality holds as $j \in E_\lambda$).  But by the
definition of ${\gb}^{i,j}_\lambda$ this contradicts $\theta
\in {\gb}^{i,j}_\lambda$.]

We now define by\ induction on $\epsilon < |{\ga}|^+$,  for  $\lambda
\in {\ga}$ (and $i<j<\kappa$),  the set ${\gb}^{i,j,\epsilon}_{\lambda}$:
\mn
\begin{enumerate}
\item[$(*)_7$]   $(\alpha) \quad {\gb}^{i,j,0}_{\lambda}
= {\gb}^{i,j}_{\lambda}$
\sn
\item[${{}}$]   $(\beta) \quad {\gb}^{i,j,\epsilon +1}_{j}
= {\gb}^{i,j,\epsilon}_{\lambda}  \cup \bigcup
\{{\gb}^{i,j,\epsilon}_{\theta}:\theta \in {\gb}^{i,j,
\epsilon}_{\lambda}\}\cup \{\theta \in {\ga}:\theta \in
\pcf({\gb}^{i,j,\epsilon}_\lambda)\}$,
\sn
\item[${{}}$]   $(\gamma) \quad {\gb}^{i,j,\epsilon}_{\lambda}
= \bigcup\limits_{{\zeta} < {\epsilon}} {\gb}^{i,j,\zeta}_{\lambda}$ for
$\epsilon < |{\ga}|^+$ limit.
\end{enumerate}
\mn
Clearly for $\lambda \in {\ga}$,  $\langle
{\gb}^{i,j,\epsilon }_{\lambda}:\epsilon < |{\ga}|^+\rangle$
belongs to  $N_{j+1}$ and is a non-decreasing sequence of subsets of
${\ga}$,  hence for some $\epsilon (i,j,\lambda) < |{\ga}|^+$,
we have

\[
[\epsilon \in (\epsilon(i,j,\lambda),|{\ga}|^+) \Rightarrow
{\gb}^{i,j,\epsilon }_{\lambda} = {\gb}^{i,j,\epsilon
(i,j,\lambda)}_{\lambda}] .
\]

\mn
So letting $\epsilon(i,j) = \sup_{\lambda \in {\ga}} \epsilon
(i,j,\lambda) < |{\ga}|^+$ we have:
\mn
\begin{enumerate}
\item[$(*)_8$]     $\epsilon(i,j) \le \epsilon < |{\ga}|^+
\Rightarrow \bigwedge\limits_{{\lambda} \in {\ga}} {\gb}^{i,j,\epsilon
(i,j)}_{\lambda} = {\gb}^{i,j,\epsilon}_{\lambda}$.
\end{enumerate}
\mn
We restrict ourselves to the case $i<j$ are from $E \backslash \varepsilon(*)$.
Which of the properties required from  $\langle {\gb}_{\lambda}:
\lambda \in {\ga}\rangle $  are satisfied by $\langle{\gb}^{i,j,
\epsilon (i,j)}_{\lambda}:\lambda \in {\ga}\rangle$?
In the conclusion of \ref{1.31} properties
$(\beta)$, $(\gamma )$ hold by the inductive
definition of  ${\gb}^{i,j,\epsilon}_{\lambda}$ (and the choice of
$\epsilon (i,j))$.  As for property $(\alpha)$, one half,
$J_{\le \lambda }[{\ga}] \subseteq  J_{<\lambda}[{\ga}] +
{\gb}^{i,j,\epsilon(i,j)}_{\lambda}$ hold by $(*)_{4}(\alpha )$ (and
${\gb}^{i,j}_{\lambda} = {\gb}^{i,j,0}_{\lambda} \subseteq
{\gb}^{i,j,\epsilon (i,j)}_{\lambda})$,  so it is enough to prove (for
$\lambda \in {\ga})$:
\mn
\begin{enumerate}
\item[$(*)_9$]    ${\gb}^{i,j,\epsilon(i,j)}_{\lambda}
\in J_{\le \lambda}[{\ga}]$.
\end{enumerate}
\mn
For this end we define by induction on $\epsilon< |{\ga}|^+$ functions
$f^{{\ga},\lambda,\epsilon}_{\alpha} $ with domain
${\gb}^{i,j,\epsilon }_{\lambda}$ for every pair
$(\alpha,\lambda)$ satisfying $\alpha < \lambda
\in {\ga}$,  such that $\zeta < \epsilon \Rightarrow
f^{{\ga},\lambda,\zeta}_{\alpha} \subseteq
f^{{\ga},\lambda,\epsilon}_{\alpha}$, so the domain increases
with $\epsilon$.

We let $f^{{\ga},\lambda,0}_{\alpha}  =
f^{{\ga},\lambda}_{\alpha} \upharpoonright {\gb}^{i,j}_{\lambda},
f^{{\ga},\lambda,\varepsilon}_{\alpha}  = \bigcup\limits_{{\zeta} <
\epsilon} f^{{\ga},\lambda,\zeta}_{\alpha}$ for limit $\epsilon <
|{\ga}|^+$ and $f^{{\ga},\lambda ,\epsilon +1}_{\alpha}$
is defined by defining each $f^{{\ga},\lambda,\epsilon+1}_{\alpha}
(\theta)$ as follows:
\medskip

\noindent
\underline{Case 1}:  If  $\theta \in {\gb}^{i,j,\epsilon}_{\lambda} $
then  $f^{{\ga},\lambda,\varepsilon +1}_\alpha(\theta) =
f^{{\ga},\lambda ,\epsilon}_{\alpha}(\theta)$.
\medskip

\noindent
\underline{Case 2}:  If $\mu \in {\gb}^{i,j,\epsilon}_{\lambda},\theta \in
{\gb}^{i,j,\epsilon }_{\mu} $ and not Case 1 and $\mu $  minimal under
those conditions, then $f^{a,\lambda,\varepsilon +1}_\alpha(\theta) =
f^{{\ga},\mu,\epsilon }_{\beta}(\theta $
where we choose  $\beta = f^{{\ga},\lambda ,\epsilon}_{\alpha}(\mu)$.
\medskip

\noindent
\underline{Case 3}:  If $\theta \in {\ga} \cap \pcf({\gb}
^{i,j,\epsilon}_{\lambda})$ and neither Case 1 nor Case 2, then

\[
f^{{\ga},\lambda,\epsilon +1}_\alpha(\theta ) = \Min\{\gamma <
\theta:f^{{\ga},\lambda ,\epsilon}_{\alpha} \upharpoonright
{\gb}_{\theta}[{\ga}] \le_{J_{<\theta }[{\ga}]} f^{{\ga},
\theta,\epsilon}_{\gamma}\}.
\]

\mn
Now $\langle\langle {\gb}^{i,j,\epsilon}_{\lambda}:\lambda\in
{\ga}\rangle:\epsilon < |{\ga}|^+\rangle$ can be computed
from ${\ga}$ and $\langle {\gb}^{i,j}_{\lambda}:\lambda \in {\ga}\rangle$.
But the latter belongs to $N_{j+1}$ by $(*)_4(\gamma)$, so the
former belongs to $N_{j+1}$ and as
$\langle \langle {\gb}^{i,j,\epsilon}_\lambda:\lambda\in {\ga}\rangle:
\epsilon < |{\ga}|^+\rangle$ is eventually constant, also each
member of the sequence belongs to $N_{j+1}$.
As also $\langle \langle f^{{\ga},\lambda }_{\alpha}: \alpha <
\lambda \rangle:\lambda \in \pcf({\ga})\rangle$ belongs to
$N_{j+1}$ we clearly get that

\[
\left\langle\langle \langle f^{{\ga},\lambda,\epsilon }_{\alpha}:
\epsilon < |{\ga}|^+\rangle:\alpha < \lambda \rangle:\lambda \in
{\ga}\right\rangle
\]

\mn
belongs to  $N_{j+1}$.
Next we prove by induction on $\epsilon$ that, for $\lambda \in {\ga}$,
we have:
\mn
\begin{enumerate}
\item[$\otimes_1$]   $\theta \in {\gb}^{i,j,\epsilon}_{\lambda}
\; \and \;  \lambda \in {\ga}
\Rightarrow f^{{\ga},\lambda,\epsilon}_{\sup(N_\kappa \cap\lambda)}
(\theta) = \sup (N_{\kappa}  \cap  \theta)$.
\end{enumerate}
\mn
For  $\epsilon = 0$ this holds by $(*)_6$.  For $\epsilon$ limit this
holds by the induction hypothesis and the definition of
$f^{{\ga},\lambda ,\epsilon}_{\alpha}$ (as union of earlier ones).
For $\epsilon  + 1$,  we check
$f^{{\ga},\lambda ,\epsilon +1}_{\sup(N_{\kappa} \cap \lambda )}(\theta )$
according to the case in its definition; for Case 1 use the induction
hypothesis applied to $f^{{\ga},\lambda,\epsilon}_{\sup(N_{\kappa}
\cap \lambda )}$.
For Case 2 (with  $\mu$),  by the induction hypothesis applied to
$f^{{\ga},\mu,\epsilon}_{\sup (N_{\kappa} \cap \mu )}$.

Lastly, for Case 3 (with  $\theta$) we should note:
\mn
\begin{enumerate}
\item[$(i)$]    ${\gb}^{i,j,\epsilon}_{\lambda} \cap
{\gb}_{\theta}[{\ga}] \notin J_{<\theta}[{\ga}]$.
\end{enumerate}
\mn
[Why?  By the case's assumption ${\gb}^{i,j,\varepsilon}_\lambda \in
(J_\theta[{\ga}])^+$ and $(*)_4(\alpha)$ above.]
\mn
\begin{enumerate}
\item[$(ii)$]    $f^{{\ga},\lambda,\epsilon}_{\sup (N_{\kappa} \cap
\lambda)} \upharpoonright ({\gb}^{i,j,\epsilon }_{\lambda}
\cap {\gb}^{i,j,\epsilon}_{\theta} ) \subseteq f^{{\ga},
\theta,\epsilon}_{\sup (N_{\kappa} \cap \theta)}$.
\end{enumerate}
\mn
[Why?  By the induction hypothesis for $\epsilon$, used concerning
$\lambda$ and $\theta$.]

Hence (by the definition in case 3 and (i) + (ii)),
\mn
\begin{enumerate}
\item[$(iii)$]    $f^{{\ga},\lambda,\epsilon +1}_{\sup (N_{\kappa} \cap
\lambda)}(\theta) \le \sup (N_{\kappa} \cap  \theta )$.
\end{enumerate}
\mn
Now if  $\gamma < \sup(N_{\kappa}  \cap  \theta )$  then for some
$\gamma(1)$ we have $\gamma < \gamma (1) \in  N_{\kappa}  \cap  \theta $,
so letting ${\gb} =: {\gb}_\lambda^{i,j,\epsilon}\cap {\gb}_\theta
[{\ga}] \cap {\gb}_\theta^{i,j,\epsilon}$, it belongs to
$J_{\le\theta}[{\ga}] \setminus J_{<\theta}[{\ga}]$ and we have

\[
f^{{\ga},\theta}_{\gamma} \upharpoonright {\gb}
<_{J_{<\theta}[{\ga}]} f^{{\ga},\theta}_{\gamma (1)}
\upharpoonright {\gb} \le f^{{\ga},\theta,\epsilon}_{\sup (N_{\kappa}
\cap \theta)}
\]

\mn
hence
$f_{\sup(N_\kappa\cap\lambda)}^{{\ga},\lambda,\epsilon+1}(\theta)
>\gamma$; as this holds for every $\gamma<\sup(N_\kappa\cap\theta)$
we have obtained
\mn
\begin{enumerate}
\item[$(iv)$]    $f^{{\ga},\lambda,\epsilon +1}
_{\sup (N_{\kappa} \cap \lambda)} (\theta ) \ge \sup(N_{\kappa} \cap \theta )$;
\end{enumerate}
\mn
together we have finished proving the inductive step for
$\epsilon+1$, hence we have proved $\otimes_{1}$.

This is enough for proving ${\gb}_\lambda^{i,j,\epsilon}\in
J_{\le\lambda}[{\ga}]$.

Why?  If it fails, as ${\gb}_\lambda^{i,j,\epsilon} \in N_{j+1}$ and
$\langle f_\alpha^{{\ga},\lambda,\epsilon}:\alpha<\lambda\rangle$
belongs to $N_{j+1}$, there is $g \in \prod
{\gb}_\lambda^{i,j,\epsilon}$ such that
\mn
\begin{enumerate}
\item[$(*)$]   $\alpha < \lambda \Rightarrow
f_\alpha^{{\ga},\lambda,\epsilon} \upharpoonright
{\gb}^{i,j,\epsilon} < g \mod J_{\le\lambda}[{\ga}]$.
\end{enumerate}
\mn
Without loss of generality $g \in N_{j+1}$; by $(*)$,
$f^{{\ga},\lambda,\epsilon}_{\sup(N_\kappa\cap\lambda)} < g \mod
J_{\le\lambda}[{\ga}]$.  But $g < \langle\sup(N_\kappa\cap\theta):
\theta\in {\gb}_\lambda^{i,j,\epsilon}\rangle$.
Together this contradicts $\otimes_1$!

This ends the proof of \ref{1.31}.
\end{PROOF}

\noindent
If $|\pcf({\ga})| < \Min({\ga})$ then \ref{1.31}
is fine and helpful.  But as we do not know this, we shall use the
following substitute.
\begin{claim}
\label{6.7A}
Assume $|{\ga}| < \kappa = \cf(\kappa)< \Min({\ga})$ and $\sigma$ is an
infinite ordinal satisfying $|\sigma|^+ < \kappa $.  Let  $\bar f$, $\bar N =
\langle N_{i}:i < \kappa \rangle$,  $N_{\kappa}$ be as in the
proof of \ref{1.31}.
\Then \, we can find $\bar i =\langle i_{\alpha}:\alpha \le \sigma
\rangle$, $\bar{\ga} = \langle{\ga}_{\alpha}:\alpha< \sigma\rangle$
and $\langle \langle {\gb}^\beta_{\lambda} [\bar{\ga}]:\lambda \in
{\ga}_{\beta} \rangle:\beta < \sigma \rangle$ such that:
\mn
\begin{enumerate}
\item[$(a)$]    $\bar i$ is  a strictly increasing continuous sequence of
ordinals $< \kappa$,
\sn
\item[$(b)$]    for  $\beta < \sigma$  we have $\langle i_{\alpha}:
\alpha \le \beta \rangle \in  N_{i_{\beta +1}}$ hence
$\langle N_{i_{\alpha}}:\alpha \le\beta \rangle \in
N_{i_{\beta +1}}$  and $\langle {\gb}^\gamma_{\lambda}[\bar{\ga}]:
\lambda \in {\ga}_{\gamma}$ and $\gamma \le \beta \rangle
\in N_{i_{\beta +1}}$, we can get $\bar i\upharpoonright (\beta+1)\in
N_{i_{\beta}+1}$ if $\kappa$ succesor of regular (we just need a
suitable partial square)
\sn
\item[$(c)$]    ${\ga}_{\beta} = N_{i_{\beta} } \cap \pcf({\ga})$,
so ${\ga}_{\beta}$ is increasing
continuous with $\beta,{\ga} \subseteq {\ga}_{\beta}
\subseteq \pcf({\ga})$ and $|{\ga}_{\beta}| < \kappa$,
\sn
\item[$(d)$]    ${\gb}^\beta_{\lambda}[\bar{\ga}]
\subseteq {\ga}_{\beta} $ (for $\lambda \in {\ga}_{\beta})$,
\sn
\item[$(e)$]    $J_{\le \lambda}[{\ga}_{\beta}] = J_{<\lambda}
[{\ga}_{\beta}] + {\gb}^\beta_{\lambda}[\bar{\ga}]$ (so
$\lambda \in {\gb}^\beta_{\lambda} [{\ga}]$
and ${\gb}^\beta_{\lambda} [\bar{\ga}] \subseteq \lambda^+)$,
\sn
\item[$(f)$]    if $\mu < \lambda$ are from ${\ga}_{\beta}$ and $\mu\in
{\gb}^\beta_{\lambda} [\bar{\ga}]$ \then \, ${\gb}^\beta_{\mu}
[\bar{\ga}] \subseteq {\gb}^\beta_{\lambda} [\bar{\ga}]$ (i.e., smoothness),
\sn
\item[$(g)$]    ${\gb}^\beta_\lambda [\bar{\ga}] =
{\ga}_{\beta} \cap \pcf({\gb}^\beta_{\lambda}[\bar{\ga}])$ (i.e., closedness),
\sn
\item[$(h)$]    if ${\gc} \subseteq {\ga}_{\beta},\beta < \sigma$
and ${\gc} \in N_{i_{\beta +1}}$ \then \, for some finite
${\gd} \subseteq {\ga}_{\beta +1} \cap \pcf({\gc})$,  we have
${\gc} \subseteq \bigcup\limits_{\mu \in {\gd}}
{\gb}^{\beta +1}_{\mu} [\bar{\ga}]$;
\end{enumerate}
\mn
more generally (note that in (h)$^+$ if $\theta = \aleph_{0}$ then we get (h)).
\mn
\begin{enumerate}
\item[$(h)^+$]    if ${\gc} \subseteq {\ga}_{\beta},
\beta < \sigma,{\gc} \in N_{i_{\beta +1}},\theta = \cf
(\theta)\in N_{i_{\beta+1}}$, \then \, for some
${\gd} \in N_{i_{\beta+1}},{\gd} \subseteq {\ga}_{\beta +1}
\cap \pcf_{\theta-\text{complete}} ({\gc})$  we have ${\gc} \subseteq
\bigcup\limits_{\mu \in {\gd}} {\gb}^{\beta+1}_{\mu}
[\bar{\ga}]$  and $|{\gd}| < \theta$,
\sn
\item[$(i)$]    ${\gb}^\beta_{\lambda}[\bar{\ga}]$
increases with $\beta$.
\end{enumerate}
\end{claim}

\noindent
This will be proved below.
\begin{claim}
\label{6.7B}
In \ref{6.7A} we can also have:
\mn
\begin{enumerate}
\item   if we let ${\gb}_{\lambda} [\bar{\ga}] =
{\gb}^\sigma_{\lambda} [{\ga}] = \bigcup\limits_{{\beta} < {\sigma}}
{\gb}^\beta_{\lambda} [\bar{\ga}]$,  ${\ga}_{\sigma}
= \bigcup\limits_{{\beta} < {\sigma}} {\ga}_{\beta}$ then also for
$\beta = \sigma $  we have {\rm (b)} (use
$N_{i_{\beta} +1})$, {\rm (c), (d), (f), (i)}
\sn
\item   If $\sigma = \cf(\sigma) > |{\ga}|$ then for $\beta =
\sigma$ also {\rm (e), (g)}
\sn
\item    If $\cf(\sigma ) > |{\ga}|,{\gc} \in N_{i_{\sigma}}$,
${\gc} \subseteq {\ga}_\sigma$ (hence  $|{\gc}| <
\Min({\gc})$ and ${\gc} \subseteq {\ga}_\sigma)$,
then for some finite ${\gd} \subseteq (\pcf({\gc})) \cap
{\ga}_{\sigma}$ we have ${\gc} \subseteq
\bigcup\limits_{\mu \in {\gd}} {\gb}_{\mu} [\bar{\ga}]$.
Similarly for $\theta $-complete,  $\theta < \cf(\sigma)$
(i.e., we have clauses {\rm (h), (h)$^+$} for $\beta=\sigma$).
\sn
\item   We can have continuity in  $\delta\le\sigma$ when
$\cf(\delta) > |{\ga}|$, i.e., ${\gb}^\delta_\lambda[\bar{\ga}]
= \bigcup\limits_{\beta<\delta} {\gb}^\beta_\lambda[\bar{\ga}]$.
\end{enumerate}
\end{claim}

\noindent
We shall prove \ref{6.7B} after proving \ref{6.7A}.
\begin{remark}
\label{6.7C}
1) If we would like to use length $\kappa$,  use
$\bar N$  as produced in \cite[2.6]{Sh:420} so
$\sigma = \kappa$.

\noindent
2)  Concerning \ref{6.7B}, in \ref{6.7C}(1) for a club $E$  of
$\sigma = \kappa$, we have $\alpha \in  E \Rightarrow
{\gb}^\alpha_\lambda [\bar{\ga}] =
{\gb}_{\lambda} [\bar{\ga}] \cap {\ga}_{\alpha}$.

\noindent
3)  We can also use \ref{6.7A},\ref{6.7B} to give an alternative
proof of part of the localization theorems similar to the one given in
the Spring '89 lectures.
\end{remark}

\noindent
For example:
\begin{claim}
\label{6.7C.1}
1) If $|{\ga}| < \theta = \cf(\theta) < \Min({\ga})$, for no sequence
$\langle \lambda_i:i < \theta \rangle$ of members of $\pcf({\ga})$, do we have
$\bigwedge\limits_{\alpha < {\theta}} [\lambda_{\alpha} >
\max \pcf\{\lambda_{i}:i < \alpha\}]$.

\noindent
2) If $|{\ga}|< \Min({\ga}),|{\gb}| <
\Min({\gb}),{\gb} \subseteq \pcf({\ga})$ and $\lambda\in \pcf({\ga})$,
\then \, for some ${\gc} \subseteq {\gb}$ we have $|{\gc}| \le |{\ga}|$
and $\lambda \in \pcf({\gc})$.
\end{claim}

\begin{PROOF}{\ref{6.7C}}
Relying on \ref{6.7A}:

\noindent
1) Without loss of generality $\Min({\ga}) > \theta^{+3}$,  let
$\kappa = \theta^{+2}$,  let $\bar N$, $N_{\kappa} $,
$\bar{\ga}$, ${\gb}$ (as a function), $\langle i_{\alpha}:
\alpha \le \sigma =: |{\ga}|^+\rangle$ be as in \ref{6.7A} but
we in addition assume that $\langle \lambda_{i}:i<\theta\rangle \in N_{0}$.
So for $j<\theta$, ${\gc}_{j} =: \{\lambda_{i}: i < j\} \in  N_{0}$
(so ${\gc}_{j} \subseteq \pcf({\ga}) \cap N_0 =
{\ga}_{0})$ hence (by clause (h) of \ref{6.7A}), for some
finite ${\gd}_{j} \subseteq {\ga}_{1} \cap \pcf({\gc}_{j}) = N_{i_{1}} \cap
\pcf({\ga}) \cap \pcf({\gc}_{j})$ we have
${\gc}_{j} \subseteq  \bigcup_{{\lambda} \in {{\gd}_{j}}}
{\gb}^1_{\lambda}[\bar{\ga}]$.  Assume  $j(1) < j(2) < \theta $.
Now if  $\mu  \in {\ga} \cap \bigcup\limits_{\lambda \in {\gd}_{j(1)}}
{\gb}^1_{\lambda} [\bar{\ga}]$  then for some
$\mu_0\in {\gd}_{j(1)}$ we have $\mu \in {\gb}^1_{\mu_{0}}
[\bar{\ga}]$;  now  $\mu_{0} \in {\gd}_{j(1)}
\subseteq  \pcf({\gc}_{j(1)}) \subseteq \pcf ({\gc}_{j(2)})
\subseteq \pcf (\bigcup\limits_{\lambda \in {\gd}_{j(2)}}{\gb}^1_{\lambda}
[\bar{\ga}]) = \bigcup\limits_{{\lambda} \in {{\gd}_{j(2)}}}(\pcf
({\gb}^1_{\lambda}[\bar{\ga}])$  hence (by clause (g) of
\ref{6.7A} as $\mu_0 \in {\gd}_{j(0)}\subseteq N_1$)
for some  $\mu_{1} \in {\gd}_{j(2)},\mu_{0} \in {\gb}^1_{\mu_{1}}
[\bar{\ga}]$.
So by clause (f) of \ref{6.7A} we have
${\gb}^1_{\mu_{0}}[\bar{\ga}] \subseteq {\gb}^1_{\mu_{1}}
[\bar{\ga}]$ hence remembering $\mu \in {\gb}_{\mu_0}^1
[\bar{\frak a}]$, we have $\mu  \in  {\gb}^1_{\mu_{1}}[\bar{\ga}]$.
Remembering $\mu$ was any member of
${\ga} \cap \bigcup\limits_{\lambda\in {\gd}_{j(1)}}
{\gb}^1_\lambda[\bar{\ga}]$,
we have ${\ga} \cap \bigcup\limits_{{\lambda \in } {\gd}{j(1)}}
{\gb}^1_{\lambda} [\bar{\ga}] \subseteq {\ga}
\cap \bigcup\limits_{{\lambda \in } {\gd} {j(2)}} {\gb}^1_{\lambda}
[\bar{\ga}]$ (holds also without ``${\ga} \cap$'' but not used).
So $\langle{\ga} \cap \bigcup\limits_{{\lambda}\in{{\gd}_{j}}}
{\gb}^1_{\lambda} [\bar{\ga}]:j < \theta \rangle$  is a
$\subseteq$-increasing sequence of subsets of ${\ga}$, but $\cf(\theta)
> |{\ga}|$,  so the sequence is eventually constant, say for $j \ge  j(*)$.
But

\begin{equation*}
\begin{array}{clcr}
\max \pcf({\ga} \cap \bigcup\limits_{\lambda \in {\gd}_j}
{\gb}^1_{\lambda} [\bar{\ga}]) & \le \max \pcf
(\bigcup\limits_{\lambda \in {\gd}_j} {\gb}^1_{\lambda}[\bar{\ga}]) \\
  &= \underset{\lambda \in {\gd}_j} \max (\max \pcf({\gb}^1_{\lambda}
[\bar{\ga}])) \\
  & = \underset{\lambda \in {\gd}_j} \max \lambda \le
\max \pcf\{\lambda_{i}:i < j\} < \lambda_{j} \\
  & = \max \pcf({\ga} \cap \bigcup\limits_{\lambda \in {\gd}_{j+1}}
{\gb}^1_{\lambda} [\bar{\ga}])
\end{array}
\end{equation*}

\mn
(last equality as ${\gb}_{\lambda_{j}}[\bar{\ga}] \subseteq
{\gb}^1_{\lambda} [\bar{\ga}] \mod J_{<\lambda}[{\ga}_{1}])$.
Contradiction.

\noindent
2) (Like \cite[\S3]{Sh:371}): If this fails choose a counterexample
$\gb$ with $|\gb|$ minimal, and among those with $\max \pcf(\gb)$ minimal
and among those with $\bigcup\{\mu^+:\mu\in\lambda\cap \pcf(\gb)\}$ minimal.
So by the pcf theorem
\mn
\begin{enumerate}
\item[$(*)_1$]   $\pcf({\gb}) \cap \lambda$ has no last member
\sn
\item[$(*)_2$]   $\mu = \sup[\lambda\cap \pcf(\gb)]$ is
not in $\pcf(\gb)$ or $\mu=\lambda$.
\sn
\item[$(*)_3$]   $\max \pcf(\gb) = \lambda$.
\end{enumerate}
\mn
Try to choose by induction on $i<|\ga|^+$, $\lambda_i\in\lambda
\cap \pcf(\gb)$, $\lambda_i > \max \pcf\{\lambda_j:j<i\}$.
Clearly by part (1), we will be stuck at some $i$.
Now $\pcf\{\lambda_j:j < i\}$ has a last member and is included in
$\pcf({\gb})$, hence by $(*)_3$ and being stuck at necessarily
$\pcf(\{\lambda_j: j<i\}) \nsubseteq \lambda$ but it is $\subseteq
\pcf(\gb)\subseteq\lambda^+$, so $\lambda = \max \pcf\{\lambda_j:j<i\}$.
For each $j$, by the choice of ``minimal counterexample" for some
$\gb_j \subseteq \gb$, we have $|\gb_j|\le |\ga|$,
$\lambda_j\in \pcf(\gb_j)$.
So $\lambda\in \pcf\{\lambda_j: j<i\}\subseteq \pcf
(\bigcup\limits_{j<i} \gb_j)$ but $\bigcup\limits_{j<i} \frb_j$ is a subset of
$\gb$ of cardinality $\le |i| \times |\ga| =|\ga|$, so we are done.
\end{PROOF}

\begin{proof}
\label{6.7D}
Without loss of generality
$\sigma = \omega \sigma$ (as we can use $\omega^\omega \sigma$ so
$|\omega^\omega \sigma| = |\sigma|$).  Let $\bar f^{\ga} = \langle
\bar f^{{\ga},\lambda} = \langle\langle
f_\alpha^{\ga,\lambda}:\alpha<\lambda\rangle:\lambda\in
\pcf(\ga)\rangle$ and $\langle N_i:i \le \kappa \rangle$
 be chosen as in the proof of \ref{1.31} and \wilog \,
$\bar f^{\ga}$ belongs to $N_0$.
For  $\zeta < \kappa $  we define $\ga^\zeta =: N_{\zeta}
\cap \pcf(\ga)$;  we also define ${}^\zeta \bar f$  as
$\langle \langle f^{\ga^\zeta,\lambda}_{\alpha}:\alpha <
\lambda \rangle: \lambda \in \pcf(\ga)\rangle$ where $f^{\ga^\zeta,
\lambda}_{\alpha} \in \prod \ga^\zeta $ is defined as follows:
\mn
\begin{enumerate}
\item[$(a)$]    if  $\theta \in \ga,f^{\ga^\zeta ,\lambda }_{\alpha}
(\theta) = f^{\ga,\lambda}_{\alpha}(\theta)$,
\sn
\item[$(b)$]    if  $\theta \in \ga^\zeta \backslash \ga$ and
$\cf(\alpha) \notin (|\ga^\zeta|,\Min(\ga)$),  then
\[
f^{\ga^\zeta,\lambda}_{\alpha}(\theta) = \Min\{\gamma <
\theta:f^{\ga,\lambda}_{\alpha} \upharpoonright \gb_{\theta}[\ga]
\le_{J_{<\theta}[\gb_{\theta}[\ga]]}f^{\ga,\theta}_\gamma
\upharpoonright \gb_{\theta}[\ga]\},
\]
\sn
\item[$(c)$]    if $\theta \in \ga^\zeta \backslash \ga$ and $\cf(\alpha)\in
(|\ga^\zeta|,\Min(\ga)$),  define $f^{\ga^\zeta ,\lambda }_{\alpha}
(\theta )$  so as to satisfy $(*)_{1}$ in the proof of \ref{1.31}.
\end{enumerate}
\mn
Now ${}^\zeta{\bar f}$ is legitimate except that we have only

\[
\beta < \gamma < \lambda \in \pcf(\ga)
\Rightarrow f^{\ga^\zeta,\lambda}_{\beta} \le
f^{\ga^\zeta,\lambda}_{\gamma} \mod J_{<\lambda}[\ga^\zeta]
\]

\mn
(instead of strict inequality) however we still have
$\bigwedge\limits_{\beta < \lambda} \bigvee\limits_{\gamma<\lambda}
[f^{\ga^\zeta,\lambda}_{\beta} < f^{\ga^\zeta ,\lambda}_{\gamma} \mod
J_{<\lambda} [\ga^\zeta]]$,  but this suffices.
(The first statement is actually proved in \cite[3.2A]{Sh:371}, the
second in \cite[3.2B]{Sh:371}; by it also ${}^\zeta{\bar f}$ is cofinal in the
required sense.)

For every $\zeta < \kappa$ we can apply the proof of \ref{1.31}
with  $(N_{\zeta} \cap \pcf(\ga))$,
$^\zeta\bar f$  and $\langle N_{\zeta +1+i}: i <
\kappa \rangle$ here standing for $\ga$, $\bar f$, $\bar N$  there.
In the proof of \ref{1.31} get a club $E^\zeta$ of $\kappa$
(corresponding to $E$ there and \wilog \, $\zeta + \Min(E^\zeta)
= \Min(E^\zeta)$ so any $i < j$ from $E^\zeta$ are O.K.).
Now we can define for $\zeta<\kappa$ and $i < j$ from $E^\zeta$,
$^\zeta \gb^{i,j}_{\lambda} $ and
$\langle {}^\zeta \gb^{i,j,\epsilon }_{\lambda}: \epsilon
< |\ga^\zeta|^+ \rangle,\langle\epsilon^\zeta(i,j,\lambda):
\lambda\in\ga^\zeta\rangle$,$\epsilon^\zeta(i,j)$, as well
as in the proof of \ref{1.31}.

Let:

\begin{equation*}
\begin{array}{clcr}
E = \{i < \kappa:&i \text{ is a limit ordinal }
(\forall j < i)(j + j < i \; \and \;  j \times  j < i) \\
  &\text{ and } \bigwedge\limits_{j < i} i \in  E^j\}.
\end{array}
\end{equation*}

\mn
So by \cite[\S1]{Sh:420} we can find $\bar C = \langle C_{\delta}:
\delta \in  S \rangle,S \subseteq \{\delta < \kappa:\cf(\delta) =
\cf(\sigma)\}$  stationary, $C_{\delta}$ a club of $\delta,
\otp(C_{\delta}) = \sigma$ such that:
\mn
\begin{enumerate}
\item   for each  $\alpha < \lambda,\{C_{\delta}  \cap  \alpha:
\alpha \in \nacc(C_{\delta})\}$ has cardinality $< \kappa$.
If $\kappa $  is successor of regular, then we
can get $[\gamma \in C_{\alpha} \cap C_{\beta}  \Rightarrow
C_{\alpha} \cap  \gamma = C_{\beta} \cap  \gamma]$
and
\sn
\item    for every club $E'$ of $\kappa$ for stationarily many
$\delta \in  S,C_{\delta} \subseteq E'$.
\end{enumerate}
\mn
Without loss of generality  $\bar C \in  N_{0}$.
For some  $\delta^*,C_{\delta^*} \subseteq  E$,  and let
$\{j_{\zeta}: \zeta \le  \omega ^2\sigma \}$  enumerate
$C_{\delta^*} \cup  \{\delta^*\}$.
So  $\langle j_{\zeta}: \zeta \le \omega^2\sigma \rangle$
is a strictly increasing continuous
sequence of ordinals from  $E \subseteq  \kappa $  such that
$\langle j_{\epsilon}:\epsilon \le \zeta\rangle\in N_{j_{\zeta+1}}$
and if, e.g., $\kappa$ is a successor of regulars then $\langle
j_\varepsilon:\varepsilon \le \zeta \rangle \in N_{j_\zeta +1}$.
Let  $j(\zeta ) = j_{\zeta}$ and for $\ell \in \{0,2\}$ let
$i_\ell(\zeta ) = i^\ell_{\zeta}  =: j_{\omega^\ell(1+\zeta)},
\ga_{\zeta} = N^\ell_{i_{\zeta}} \cap \pcf(\ga)$, and
$\bar\ga^\ell =: \langle \ga^\ell_{\zeta}:\zeta < \sigma \rangle,
{}^\ell \gb^\zeta_{\lambda}[\bar \ga] =: {}^{i_\ell(\zeta)}
\gb_\lambda^{j(\omega^\ell\zeta +1),j(\omega^\ell\zeta +2),
\epsilon^\zeta (j(\omega^\ell\zeta +1), j(\omega^\ell\zeta +2))}$.
Recall that $\sigma = \omega \sigma$ so $\sigma = \omega^2 \sigma$; if
the value of $\ell$ does not matter we omit it.
Most of the requirements follow immediately by the proof of
\ref{1.31}, as
\mn
\begin{enumerate}
\item[$\circledast$]     for each $\zeta<\sigma$, we have $\gb_{\zeta} $,
$\langle \gb^\zeta_{\lambda}
[\bar \ga]:\lambda \in \ga_{\zeta} \rangle$  are as in the proof
(hence conclusion of \ref{1.31}) and belongs to $N_{i_{\beta} +3}
\subseteq N_{i_{\beta +1}}$.
\end{enumerate}
\mn
We are left (for proving \ref{6.7A}) with proving clauses
(h)$^+$ and (i) (remember that
(h) is a special case of (h)$^+$ choosing $\theta=\aleph_0)$.

For proving clause (i) note that for  $\zeta < \xi  < \kappa $,
$f^{\ga^\zeta,\lambda}_{\alpha} \subseteq f^{\ga^\xi,\lambda}_{\alpha}$ hence
${}^\zeta \gb^{i,j}_{\lambda}  \subseteq {}^\xi \gb^{i,j}_{\lambda}$.
Now we can prove by induction on  $\epsilon$  that
${}^\zeta \gb^{i,j, \epsilon }_{\lambda}  \subseteq
{}^\xi \gb^{i,j,\epsilon}_{\lambda}$ for every $\lambda \in \ga_\zeta$
(check the definition in $(*)_7$ in the proof of \ref{1.31}) and the
conclusion follows.

Instead of proving (h)$^+$ we prove an apparently weaker version
(h)$'$ below, but having $(h)'$ for the case $\ell=0$ gives $(h)^+$
for $\ell=2$ so this is enough [[then note that $\bar i'=\langle
i_{\omega^2\zeta}:\zeta<\sigma\rangle$, $\bar\ga' =\langle
\ga_{\omega^2\zeta}:\zeta<\sigma\rangle$, $\langle
N_{i(\omega^2\zeta)}:\zeta<\sigma\rangle$,
$\langle \gb_\lambda^{\omega^2\zeta} [\bar\ga']:\zeta<\sigma,
\lambda\in \ga'_\zeta = \ga_{\omega^2\zeta}\rangle$ will exemplify the
conclusion]] where:
\mn
\begin{enumerate}
\item[$(h)'$]    if $\gc \subseteq \ga_\beta$, $\beta<\sigma$,
$\gc \in N_{i_{\beta+1}},\theta = \cf(\theta)\in N_{i_{\beta+1}}$
\then \, for some $\frd\in N_{i_{\beta+\omega+1}+1}$ satisfying
$\gd \subseteq \ga_{\beta+\omega} \cap \pcf_{\theta-\text{complete}}(\gc)$ we
have $\gc \subseteq \bigcup\limits_{\mu \in \gd} \gb_\mu^{\beta+\omega}
[\bar\ga]$ and $|\gd|<\theta$.
\end{enumerate}
\end{proof}

\begin{PROOF}{\ref{6.7A}}
\underline{Proof of $(h)'$}

So let $\theta,\beta,\gc$ be given; let $\langle \gb_{\mu}[\bar\ga]:
\mu\in \pcf(\gc)\rangle (\in N_{i_{\beta+1}})$ be a generating sequence.
We define by induction on  $n < \omega $,  $A_{n}$,
$\langle(\gc_{\eta},\lambda_{\eta}):\eta \in A_{n}\rangle$
such that:
\mn
\begin{enumerate}
\item[$(a)$]    $A_0=\{\langle\rangle\},\gc_{\langle \rangle} = \gc$,
$\lambda_{\langle \rangle} = \max \pcf(\gc)$,
\sn
\item[$(b)$]    $A_{n} \subseteq {}^n \theta,|A_{n}| < \theta $,
\sn
\item[$(c)$]    if  $\eta\in A_{n+1}$ then $\eta\upharpoonright n\in A_{n}$,
$\gc_{\eta} \subseteq \gc_{\eta\upharpoonright n}$, $\lambda_{\eta} <
\lambda_{\eta \upharpoonright n}$ and $\lambda_{\eta} = \max \pcf
(\gc_{\eta})$,
\sn
\item[$(d)$]    $A_{n},\langle(\gc_{\eta},\lambda_{\eta}):\eta \in
A_{n}\rangle $  belongs to  $N_{i_{\beta +1+n}}$ hence  $\lambda_{\eta}
\in  N_{i_{\beta +1+n}}$,
\sn
\item[$(e)$]    if  $\eta \in A_{n}$ and  $\lambda_{\eta}  \in
\pcf_{\theta \text{-complete}}(\gc_{\eta})$ and $\gc_{\eta}
\nsubseteq \gb^{\beta +1+n}_{\lambda_{\eta} }[\bar\ga]$  then
$(\forall \nu)[\nu  \in A_{n+1} \;  \and \;  \eta \subseteq \nu
\Leftrightarrow \nu  = \eta \char94 \langle 0\rangle ]$  and
$\gc_{\eta \char94 \langle 0\rangle } =
\gc_{\eta} \backslash \gb^{\beta +1+n}_{\lambda_\eta}[\bar\ga]$  (so
$\lambda_{\eta \char94 \langle 0\rangle} = \max \pcf
(\gc_{\eta \char94 \langle 0\rangle}) < \lambda_{\eta}
= \max \pcf(\gc_{\eta})$,
\sn
\item[$(f)$]    if  $\eta \in A_{n}$ and $\lambda_{\eta}  \notin
\pcf_{\theta \text{-complete}}(\gc_\eta)$  then
\[
\gc_{\eta} = \bigcup \{\gb_{\lambda_{\gamma \char 94 \langle
i\rangle}}[\gc]: i < i_n < \theta,\eta \char 94 \langle i \rangle
\in A_{n+1}\},
\]
and if $\nu  = \eta \char 94 \langle i\rangle \in  A_{n+1}$
then  $\gc_{\nu}  = \gb_{\lambda_{\nu}}[\gc]$,
\sn
\item[$(g)$]    if $\eta\in A_n$, and $\lambda_\eta \in
\pcf_{\theta \text{-complete}} (\gc_\eta)$ but
$\gc_\eta \subseteq \gb_{\lambda_n}^{\beta+1-n}[\bar\ga]$, then
$\neg(\exists\nu)[\eta\triangleleft\nu\in A_{n+1}]$.
\end{enumerate}
\mn
There is no problem to carry the definition (we use \ref{6.7F}(1),
the point is that $\gc \in  N_{i_{\beta +1+n}}$ implies
$\langle \gb_{\lambda}(\gc):\lambda \in \pcf_{\theta}
[\gc]\rangle \in N_{i_{\beta +1+n}}$ and as there
is $\gd$ as in \ref{6.7F}(1), there is one in $N_{i_{\beta +1+n+1}}$ so
$\gd \subseteq \ga_{\beta +1+n+1}$).

Now let

\[
\gd_n =: \{\lambda_{\eta}^{}:\eta \in A_{n} \text{ and }\lambda_{\eta}  \in
\pcf_{\theta \text{-complete}}(\gc_\eta)\}
\]

\mn
and $\gd =: \bigcup\limits_{n<\omega} \gd_n$;
we shall show that it is as required.

The main point is $\gc \subseteq \bigcup\limits_{\lambda \in {\gd}}
\gb^{\beta +\omega }_{\lambda} [\bar \ga]$; note that

\[
[\lambda_{\eta}  \in \gd,\eta \in  A_n  \Rightarrow
\gb^{\beta +1+n}_{\lambda_{\eta}} [\bar \ga]
\subseteq \gb^{\beta +\omega }_{\lambda_{\eta}}[\bar\ga]]
\]

\mn
hence it suffices to show  $\gc \subseteq \bigcup\limits_{n<\omega}
\bigcup\limits_{\lambda \in  \gd_n} \gb^{\beta + 1 + n}_{\lambda}
[\bar\ga]$, so assume  $\theta \in \gc \backslash
\bigcup\limits_{n<\omega} \bigcup\limits_{{\lambda} \in {\gd_n}} \gb^{\beta
+ 1 + n}_{\lambda} [\bar \ga]$,  and we choose by induction on  $n$,
$\eta_{n} \in  A_{n}$ such that  $\eta_{0} = <>$, $\eta_{n+1}
\upharpoonright n = \eta_{n}$ and $\theta \in
\gc_{\eta} $;  by clauses (e) $+$ (f) above
this is possible and $\langle \max \pcf(\gc_{\eta_{n}}):n <
\omega \rangle$ is (strictly) decreasing, contradiction.

The minor point is  $|\gd| < \theta$;  if  $\theta > \aleph_{0}$ note
that $\bigwedge\limits_{n} |A_{n}| < \theta$  and $\theta
= \cf(\theta)$ clearly $|\gd| \le | \bigcup_{n} A_{n}| < \theta +
\aleph_{1} = \theta$.

If  $\theta = \aleph_{0}$ (i.e.\ clause (h)) we should show that
$\bigcup\limits_{n} A_{n}$ finite; the proof is as above noting that the
clause (f) is vacuous now.  So $n < \omega \Rightarrow |A_n| =1$ and
for some $n \bigvee\limits_{n} A_n=\emptyset$, so
$\bigcup\limits_{n} A_n$ is finite.
Another minor point is $\gd \in N_{i_{\beta+\omega+1}}$; this
holds as the construction is unique from ${\gc},\langle
{\gb}_\mu[{\gc}]:\mu \in \pcf({\gc})\rangle,
\langle N_j:j< i_{\beta+\omega}\rangle,
\langle i_j:j\le \beta+\omega\rangle$, $\langle
(\ga_{i(\zeta)},\langle \gb_\lambda^\zeta[\bar{\frak a}]:\lambda\in
\ga_{i(\zeta)}\rangle): \zeta\le\beta+\omega\rangle$; no
``outside'' information is used so $\langle (A_n, \langle
(c_\eta, \lambda_\eta):\eta\in A_n\rangle):n<\omega\rangle
\in N_{i_{\beta+\omega+1}}$, so (using a choice function) really $\gd \in
N_{i_{\beta+\omega+1}}$.
\end{PROOF}

\begin{PROOF}{\ref{6.7B}}
Let $\gb_\lambda [\bar \ga] =
\gb_\lambda^\sigma = \bigcup\limits_{\beta<\sigma}
\gb^\beta_\lambda [\ga_\beta]$ and $\ga_\sigma =
\bigcup\limits_{\zeta<\sigma} \ga_\zeta$.  Part (1) is straightforward.
For part (2), for clause (g), for $\beta=\sigma$, the inclusion
``$\subseteq$'' is straightforward; so assume $\mu\in
\ga_\beta \cap \pcf(\gb^\beta_\lambda[\bar\ga])$.
Then by \ref{6.7A}(c) for some $\beta_0<\beta$, we have
$\mu \in \ga_{\beta_0}$, and by \ref{6.7C.1} (which
depends on \ref{6.7A} only) for some $\beta_1<\beta$,
$\mu \in \pcf(\gb^{\beta_1}_\lambda[\bar\ga])$;
by monotonicity \wilog \, $\beta_0=\beta_1$, by clause (g) of \ref{6.7A}
applied to $\beta_0$, $\mu \in \gb^{\beta_0}_\lambda[\bar\ga]$.
Hence by clause (i) of \ref{6.7A},
$\mu\in \gb_\lambda^\beta[\bar\ga]$, thus proving the other inclusion.

The proof of clause (e) (for \ref{6.7B}(2)) is similar, and also
\ref{6.7B}(3).  For \ref{6.7(B)(4)} for
$\delta<\sigma,\cf(\delta)>|\ga|$ redefine
$\gb_\lambda^\delta [\bar{\ga}]$ as $\bigcup\limits_{\beta<\delta}
\gb_\lambda^{\beta+1}[\bar\ga]$.
\end{PROOF}

\begin{claim}
\label{6.7F}
Let  $\theta $  be regular.

\noindent
0)   If  $\alpha < \theta,\pcf_{\theta \text{-complete}}
(\bigcup\limits_{{i} < {\alpha}} \ga_{i})  = \bigcup\limits_{{i} < {\alpha}}
\pcf_{\theta \text{-complete}}(\ga_{i})$.

\noindent
1)  If $\langle \gb_{\partial} [\ga]:\partial \in \pcf(\ga)
\rangle$ is a generating sequence for $\ga$,
$\gc \subseteq \ga$, \then \, for some $\gd \subseteq
\pcf_{\theta \text{-complete}}(\gc)$ we have:
$|\gd|<\theta$ and $\gc \subseteq \bigcup\limits_{{\theta} \in
{\ga}} \gb_{\theta} [\ga]$.

\noindent
2) If $|\ga \cup \gc| < \Min(\ga),\gc \subseteq
\pcf_{\theta \text{-complete}}(\ga)$,  $\lambda \in
\pcf_{\theta \text{-complete}}(\gc)$  then $\lambda \in
\pcf_{\theta \text{-complete}}(\ga)$.

\noindent
3)  In (2) we can weaken $|\ga \cup \gc| < \Min(\ga)$ to
$|\ga| < \Min(\ga),|\gc| < \Min(\gc)$.
\end{claim}

\begin{PROOF}{\ref{6.7F}}
(0) and (1): Left to the reader.

\noindent
2)  See \cite[1.10--1.12]{Sh:345b}.

\noindent
3)  Similarly.
\end{PROOF}

\begin{claim}
\label{6.7G}
1) Let $\theta$ be regular $\le |{\ga}|$.
We cannot find $\lambda_{\alpha} \in \pcf_{\theta \text{-complete}}(\ga)$ for
$\alpha < |\ga|^+$ such that  $\lambda_{i} > \sup
\pcf_{\theta \text{-complete}}(\{\lambda_{j}:j < i\})$.

\noindent
2)  Assume $\theta \le |\ga|,\gc \subseteq
\pcf_{\theta \text{-complete}}(\ga)$  (and $|\gc| <
\Min(\gc)$;  of course  $|\ga| < \Min(\ga))$.
If  $\lambda \in \pcf_{\theta \text{-complete}}(\gc)$
\then \, for some $\gd \subseteq \gc$ we have $|\gd| \le |\ga|$
and $\lambda \in \pcf_{\theta \text{-complete}}(\gd)$.
\end{claim}

\begin{PROOF}{\ref{6.7G}}
1) If  $\theta = \aleph_{0}$ we already know it (see
\ref{6.7C.1}), so assume  $\theta > \aleph_{0}$.
We use \ref{6.7A} with  $\{\theta$, $\langle \lambda_{i}: i < |\ga|^+
\rangle\}\in  N_{0}$, $\sigma = |\ga|^+$,  $\kappa = |\ga|^{+3}$
where, without loss of generality, $\kappa < \Min(\ga)$.
For each  $\alpha < |\ga|^+$ by (h)$^+$ of \ref{6.7A} there is
$\ga_{\alpha} \in  N_{i_{1}},\gd_{\alpha}  \subseteq
\pcf_{\theta \text{-complete}}(\{\lambda_{i}:i < \alpha \})$,
$|\gd_{\alpha} | < \theta$ such that $\{\lambda_{i}:i < \alpha \} \subseteq
\bigcup\limits_{{\theta} \in {\gd_{\alpha}} } \gb^1_{\theta} [\bar\ga]$;
hence by clause (g) of \ref{6.7A} and part (0) Claim \ref{6.7F} we have
$\ga_{1} \cap \pcf_{\theta \text{-complete}}(\{\lambda_{i}:
i < \alpha \}) \subseteq \bigcup\limits_{{\theta} \in {\gd_{\alpha}} }
\gb^1_{\theta} [\bar \ga]$.
So for $\alpha < \beta < |\ga|^+$,  $\gd_{\alpha}  \subseteq
\ga_1 \cap \pcf_{\theta \text{-complete}}\{\lambda_{i}:i < \alpha \}
\subseteq \ga_1\cap \pcf_{\theta \text{-complete}}
\{\lambda_{i}:i < \beta \} \subseteq \bigcup\limits_{{\theta} \in
{\gd_{\beta}}} \gb^1_{\theta} [\bar \ga]$.
As the sequence is smooth (i.e., clause (f) of \ref{6.7A}) clearly
$\alpha < \beta \Rightarrow \bigcup\limits_{{\mu} \in  {\gd_{\alpha}} }
\gb^1_{\mu} [\bar \ga] \subseteq \bigcup\limits_{{\mu} \in {\gd_{\beta}}}
\gb^1_{\mu} [\bar \ga]$.

So  $\langle \bigcup\limits_{{\mu} \in {\gd_{\alpha}} } \gb^1_{\mu}
[\bar \ga] \cap  \ga: \alpha < |\ga|^+\rangle $  is a
non-decreasing sequence of subsets of $\ga$ of length $|\ga|^+$,
 hence for some  $\alpha(*) < |\ga|^+$ we have:
\mn
\begin{enumerate}
\item[$(*)_1$]     $\alpha (*) \le \alpha < |\ga|^+
\Rightarrow \bigcup\limits_{{\mu} \in  {\gd_{\alpha}} } \gb^1_{\mu}
[\bar \ga] \cap \ga = \bigcup\limits_{{\mu} \in {\gd_{\alpha(*)}}}
\gb^1_{\mu} [\bar \ga] \cap \ga$.
\end{enumerate}
\mn
If $\tau \in \ga_{1} \cap \pcf_{\theta \text{-complete}}
(\{\lambda_{i}: i<\alpha\})$ then $\tau\in
\pcf_{\theta \text{-complete}}(\ga)$ (by parts (2),(3) of
Claim \ref{6.7F}), and $\tau \in \gb^1_{\mu_{\tau} }[\bar \ga]$
for some $\mu_{\tau}\in \gd_{\alpha} $ so $\gb^1_{\tau} [\bar \ga]
\subseteq \gb^1_{\mu_{\tau}}[\bar\ga]$, also $\tau \in
\pcf_{\theta \text{-complete}}(\gb^1_{\tau}[\bar \ga] \cap \ga)$
(by clause (e) of \ref{6.7A}),  hence

\begin{equation*}
\begin{array}{clcr}
\tau \in \pcf_{\theta \text{-complete}}(\gb^1_{\tau} [\bar \ga] \cap
\ga) &\subseteq \pcf_{\theta \text{-complete}}(\gb^1_{\mu_{\tau}}
[\bar \ga] \cap \ga) \\
  &\subseteq \pcf_{\theta \text{-complete}}
(\bigcup\limits_{{\mu} \in {\gd_{\alpha}}} \gb^1_{\mu} [\bar\ga] \cap \ga).
\end{array}
\end{equation*}

\mn
So $\ga_{1} \cap \pcf_{\theta \text{-complete}}(\{\lambda_{i}:
i < \alpha\}) \subseteq \pcf_{\theta \text{-complete}}
(\bigcup\limits_{{\mu} \in  {\gd_{\alpha}} } \gb^1_{\mu}[\bar \ga]
\cap \ga)$.
But for each  $\alpha < |\ga|^+$ we have $\lambda_{\alpha} > \sup
\pcf_{\theta \text{-complete}}(\{\lambda_{i}:i < \alpha\})$,
whereas $\gd_\alpha \subseteq \pcf_{\sigma \text{-complete}}
\{\lambda_i: i<\alpha\}$, hence  $\lambda_{\alpha} > \sup \gd_{\alpha}$ hence
\mn
\begin{enumerate}
\item[$(*)_2$]    $\lambda_{\alpha} > \sup_{\mu \in \gd_\alpha}
\max \pcf(\gb^1_{\mu} [\bar \ga]) \ge  \sup
\pcf_{\theta \text{-complete}}(\bigcup\limits_{\mu \in \gd_\alpha}
\gb^1_{\mu} [\bar\ga] \cap \ga)$.

On the other hand,
\sn
\item[$(*)_3$]    $\lambda_{\alpha} \in \pcf_{\theta
\text{-complete}}\{\lambda_{i}:i < \alpha + 1\} \subseteq
\pcf_{\theta \text{-complete}}(\bigcup\limits_{{\mu} \in
{\gd_{\alpha +1}}} \gb^1_{\mu} [\bar \ga] \cap \ga)$.
\end{enumerate}
\mn
For $\alpha = \alpha(*)$ we get contradiction by $(*)_{1}
+ (*)_{2} + (*)_{3}$.

\noindent
2)  Assume $\ga,\gc,\lambda$ form a counterexample with $\lambda$ minimal.
Without loss of generality  $|\ga|^{+3} < \Min(\ga)$ and
$\lambda = \max \pcf(\ga)$ and $\lambda = \max \pcf(\gc)$
(just let $\ga' =: \gb_{\lambda} [\ga],\gc' =: \gc \cap
\pcf_\theta[\ga']$;  if  $\lambda \notin
\pcf_{\theta \text{-complete}}(\gc')$ then necessarily
$\lambda \in \pcf(\gc\backslash \gc')$
(by \ref{6.7F}(0)) and similarly $\gc \backslash \gc' \subseteq
\pcf_{\theta \text{-complete}}(\ga \backslash \ga')$
hence by parts (2),(3) of Claim \ref{6.7F} we have
$\lambda \in \pcf_{\theta \text{-complete}}(\ga \backslash \ga')$,
contradiction).

Also without loss of generality
$\lambda \notin \gc$. Let $\kappa,\sigma,\bar N,
\langle i_{\alpha}=i(\alpha):
\alpha\le \sigma\rangle,\bar\ga =
\langle\ga_i:i \le \sigma\rangle$
be as in \ref{6.7A} with $\ga \in N_{0},\gc
\in N_{0},\lambda \in N_{0},\sigma = |\ga|^+,
\kappa = |\ga|^{+3} < \Min(\ga)$.
We choose by induction on $\epsilon < |\ga|^+,
\lambda_{\epsilon},\gd_{\epsilon}$ such that:
\mn
\begin{enumerate}
\item[$(a)$]"  $\lambda_{\epsilon} \in
\ga_{\omega^2\epsilon +\omega +1},
\gd_{\epsilon} \in N_{i(\omega^2
\epsilon +\omega +1)}$,
\sn
\item[$(b)$]   $\lambda_{\epsilon} \in \gc$,
\sn
\item[$(c)$]   $\gd_{\epsilon} \subseteq
\ga_{\omega^2 \epsilon +\omega+1}
\cap \pcf_{\theta \text{-complete}}(\{\lambda_{\zeta}:
\zeta < \epsilon\})$,
\sn
\item[$(d)$]   $|\gd_{\epsilon} | < \theta $,
\sn
\item[$(e)$]   $\{\lambda_{\zeta}:
\zeta < \epsilon \} \subseteq
\bigcup\limits_{{\theta} \in
{\gd_{\epsilon}}} \gb^{\omega ^2\epsilon
+\omega +1}_{\theta} [\bar \ga]$,
\sn
\item[$(f)$]   $\lambda_{\epsilon}  \notin
\pcf_{\theta \text{-complete}}
(\bigcup\limits_{{\theta} \in
{\gd_{\epsilon}}}
\gb^{\omega^2 \epsilon +\omega +1}_{\theta}
[\bar \ga])$.
\end{enumerate}
\mn
For every $\epsilon < |\ga|^+$ we first choose
$\gd_{\epsilon}$ as the $<^*_{\chi}$-first
element satisfying (c) + (d) + (e) and
then if possible $\lambda_{\epsilon}$
as the $<^*_{\chi}$-first element
satisfying (b) + (f).
It is easy to check the requirements and in fact
$\langle \lambda_{\zeta}:\zeta < \epsilon
\rangle \in N_{\omega^2\epsilon +1},
\langle \gd_{\zeta}:\zeta < \epsilon \rangle
\in  N_{\omega^2\epsilon +1}$
(so clause (a) will hold).
But why can we choose at all?  Now
$\lambda \notin
\pcf_{\theta \text{-complete}}\{\lambda_{\zeta}:\zeta
< \epsilon\}$ as $\ga,\gc,\lambda$ form a
counterexample with $\lambda$ minimal and
$\epsilon < |\ga|^+$ (by \ref{6.7F}(3)).
As  $\lambda = \max \pcf(\ga)$ necessarily
$\pcf_{\theta \text{-complete}}
(\{\lambda_{\zeta}:\zeta < \epsilon\})
\subseteq  \lambda$ hence
$\gd_{\epsilon} \subseteq
\lambda$ (by clause (c)).
By part (0) of Claim \ref{6.7F} (and clause
(a)) we know:

\begin{equation*}
\begin{array}{clcr}
\pcf_{\theta \text{-complete}}[\bigcup\limits_{{\mu} \in
{\gd_{\epsilon}}} \gb^{\omega^2\epsilon +
\omega +1}_{\mu}[\bar\ga]] &=
\bigcup\limits_{{\mu} \in  {\gd_{\epsilon}}}
\pcf_{\theta \text{-complete}}
[\gb^{\omega ^2+\omega +1}_{\mu}
[\bar \ga]] \\
 & \subseteq \bigcup\limits_{\mu \in
\gd_\epsilon}(\mu  + 1) \subseteq  \lambda
\end{array}
\end{equation*}

\mn
(note $\mu = \max \pcf(\gb^\beta_{\mu}
[\bar \ga]))$.
So $\lambda \notin \pcf_{\theta \text{-complete}}
(\bigcup\limits_{\mu \in \gd_\epsilon}
\gb^{\omega^2\epsilon +\omega +1}_{\mu} [\bar\ga])$   hence by part (0) of Claim \ref{6.7F}
$\gc \nsubseteq \bigcup\limits_{\mu \in
\gd_\epsilon} \gb^{\omega^2\epsilon +
\omega +1}_{\mu}
[\bar \ga]$ so $\lambda_{\epsilon} $ exists.
Now $\gd_\epsilon$ exists by \ref{6.7A}
clause (h)$^+$.

Now clearly $\left\langle \ga \cap
\bigcup\limits_{\mu \in \gd_\epsilon}
\gb^{\omega^2\epsilon +\omega +1}_{\mu}
[\bar \ga]:\epsilon < |\ga|^+\right\rangle$
is non-decreasing (as in the earlier proof)
hence eventually constant, say for
$\epsilon \ge \epsilon(*)$  (where
$\epsilon(*) < |\ga|^+)$.

But
\mn
\begin{enumerate}
\item[$(\alpha)$]  $\lambda_\epsilon \in
\bigcup\limits_{\mu \in \gd_{\epsilon +1}}
\gb^{\omega^2\epsilon +\omega +1}_{\mu}
[\bar \ga]$ [clause (e) in the choice of
$\lambda_\epsilon,\gd_\epsilon]$,
\sn
\item[$(\beta)$]   $\gb^{\omega^2\epsilon
+\omega +1}_{\lambda_{\epsilon}}[\bar \ga]
\subseteq \bigcup_{{\mu} \in
{\gd_{\epsilon +1}}} \gb^{\omega^2
\epsilon +\omega +1}_{\mu} [\bar \ga]$
[by clause (f) of \ref{6.7A} and $(\alpha)$ alone],
\sn
\item[$(\gamma)$]   $\lambda_{\epsilon} \in
\pcf_{\theta \text{-complete}}(\ga)$
[as $\lambda_{\epsilon} \in \gc$
and a hypothesis],
\sn
\item[$(\delta)$]   $\lambda_{\epsilon} \in
\pcf_{\theta \text{-complete}}
(\gb^{\omega^2\epsilon + \omega +1}
_{\lambda_{\epsilon}}[\bar \ga])$
[by $(\gamma )$ above
and clause (e) of \ref{6.7A}],
\sn
\item[$(\epsilon)$]  $\lambda_\epsilon
\notin \pcf(\ga \setminus
\gb^{\omega^2\epsilon+\omega+1}
_{\lambda_\epsilon})$,
\sn
\item[$(\zeta)$]   $\lambda_{\epsilon}  \in
\pcf_{\theta \text{-complete}}(\ga \cap
\bigcup\limits_{{\mu} \in
{\gd_{\epsilon +1}}}\gb^{\omega^2\epsilon +
\omega +1}_{\mu}[\bar \ga])$
[by $(\delta) + (\epsilon) +(\beta )]$.
\end{enumerate}
\mn
But for $\epsilon = \epsilon(*)$, the
statement $(\zeta)$ contradicts the choice
of $\epsilon(*)$ and clause (f) above.
\end{PROOF}
\newpage

\section {}

\begin{definition}
\label{cv.1}
1) For $J$ an ideal on $\kappa$ (or any set,
$\Dom(J)$-does not matter) and singular $\mu$
(usually $\cf(\mu) \le \kappa$, otherwise the
result is 0)
\mn
\begin{enumerate}
\item[$(a)$]  we define $\pp_J(\mu)$ as

\begin{equation*}
\begin{array}{clcr}
\sup\{\tcf(\prod\limits_{i <\kappa}
\lambda_i,<_J): &\lambda_i
\in \Reg \cap \mu \backslash \kappa^+
\text{ for } i < \kappa \\
  &\text{ and } \mu = \lim_J \langle
\lambda_i:i < \kappa \rangle, \text{ see \ref{cv.1a}(1) and} \\
  &(\prod\limits_{i < \kappa} \lambda_i,<_J)
\text{ has true  cofinality}\}
\end{array}
\end{equation*}
\sn
\item[$(b)$]  we define $\pp^+_J(\mu)$ as

\begin{equation*}
\begin{array}{clcr}
\sup\{(\tcf(\prod\limits_{i <\kappa}
\lambda_i,<_J))^+:&\lambda_i
\in \Reg \cap \mu \backslash
\kappa^+ \text{ for } i < \kappa \\
  &\text{ and } \mu = \lim_J(\langle
\lambda_i:i < \kappa\rangle),
\text{ see \ref{cv.1a}(1) below and} \\
  &(\prod\limits_{i < \kappa}
\lambda_i,<_J) \text{ has true cofinality}\}.
\end{array}
\end{equation*}
\end{enumerate}
\mn
2) For $\bold J$ a family of ideals on
(usually but not necessarily on
the same set) and singular $\mu$
let $\pp_{\bold J}(\mu) =
\sup\{\pp_J(\mu):J \in \bold J\}$ and
$\pp^+_J(\mu) = \sup\{\pp^+_J(\mu):J \in
\bold J\}$.

\noindent
3) For a set ${\ga}$ of regular cardinals let
$\pcf_J({\ga}) =
\{\tcf(\prod\limits_{t \in \Dom(J)}
\lambda_t,<_J):\lambda_t \in {\ga}$ for
$t \in \Dom(J)\}$;
similarly $\pcf_{\bold J}({\ga})$.
\end{definition}

\begin{remark}
\label{cv.1a}
1) Recall that $\mu = \lim_J\langle \lambda_t:
t \in \Dom(J)\rangle$, where $J$ is an
ideal on $\Dom(J)$ mean that for every
$\mu_1 < \mu$ the set $\{t \in \Dom(J):
\lambda_t \notin (\mu_1,\mu]\}$ belongs to $J$.

\noindent
2) On $\pcf_J({\ga})$: check consistency of
notation by \cite{Sh:g}.
\end{remark}

\begin{observation}
\label{cv.2}
1) For $\mu,J$ as in clause (a)
\ref{cv.1}, the following are equivalent
\mn
\begin{enumerate}
\item[$(a)$]   pp$_J(\mu) > 0$
\sn
\item[$(b)$]  the sup is on a non-empty set
\sn
\item[$(c)$]  there is an increasing
sequence of length $\cf(\mu)$
of member of $J$ whose union is $\kappa$
\sn
\item[$(d)$]  $\pp_J(\mu) > \mu$
\sn
\item[$(e)$]   every cardinal appearing in
the $\sup$ is regular $> \mu$ and the set
of those appearing is $\Reg \cap
[\mu^+,\pp^+_J(\mu))$ and is non-empty.
\end{enumerate}
\end{observation}

\begin{definition}
\label{cv.3}
1) Assume $J$ is an ideal on
$\kappa,\sigma = \cf(\sigma) \le \kappa,
f \in {}^\kappa \Ord$ \then \, we let

\begin{equation*}
\begin{array}{clcr}
\bold W_{J,\sigma}(f^*,< \mu) =
\Min\{|{\cP}|: &{\cP}
\text{ is a family of subsets of }
\sup \Rang(f^*)+1 \\
  &\text{each of cardinality } <
\mu \text{ and for every } f \le f^*,   \\
  &\Rang(f) \text{ is the union of } < \sigma \\
  &\text{sets of the form} \\
  &\{i < \kappa:f(i) \in A\},A \in {\cP}\}.
\end{array}
\end{equation*}

\mn
2) If $f^*$ is constantly $\lambda$ we
write $\lambda$ if $\mu =
\lambda$ we can omit $< \mu$.
\end{definition}

\begin{remark}
\label{cv.4}
1) See $\cov(\lambda,\mu,\theta,\sigma) =
\bold W_{[\theta]^{< \sigma},\sigma}
(\langle \lambda:i < \theta \rangle,\mu)$.

\noindent
2) On the case of normal ideals, i.e. $\prc$
see \cite[\S1]{Sh:410} and
more generally $\prd$ see \cite{Sh:410}.
\end{remark}

\noindent
We may use several families of ideals.
\begin{definition}
\label{cv.5}
Let
\mn
\begin{enumerate}
\item[$(a)$]   $\com_{\theta,\sigma} = \{J:J$ is a
$\sigma$-complete ideal on $\theta\}$
\sn
\item[$(b)$]   $\nor_\kappa = \{J:J$ a
normal ideal on $\kappa\}$
\sn
\item[$(c)$]  $\com_{I,\sigma} =
\{J:J$ is a $\sigma$-complete ideal on
$\Dom(I)$ extending the ideal $I\}$
\sn
\item[$(d)$]  $\nor_I = \{J:J$ is a normal ideal on
$\Dom(I)$ extending the ideal $I\}$.
\end{enumerate}
\end{definition}

\begin{claim}
\label{cv.7}
The $(< \aleph_1)$-covering lemma.

Assume $\aleph_1 \le \sigma \le \,\cf(\mu) \le
\kappa < \mu$ and $I$ is a $\sigma$-complete
ideal on $\kappa$.

\Then
\mn
\begin{enumerate}
\item[$(a)$]  $\bold W_{I,\sigma}(\mu) =
\pp_{\com_\sigma(I)}(\mu)$
\sn
\item[$(b)$]  except when
$\circledast_{\mu,I,\sigma}$ below holds,
we can strengthen the equality in clause (a)
to: i.e., if $\pp_{\com_\sigma(I)}$ is a
regular cardinal (so $> \mu$) then
the $\sup$ in \ref{cv.1}(1) is obtained
\sn
\begin{enumerate}
\item[$\circledast_{\mu,I,\sigma}$]
$(a) \quad \lambda =:
\pp_{\com_\sigma(I)}(\mu)$ is (weakly)
inaccessible,
the $\sup$ is not obtained and for some
set ${\ga} \subseteq \Reg \cap \mu,|{\ga}| +
\kappa < \Min({\ga})$ and $\lambda =
\sup(\pcf_{I,\sigma}({\ga}))$; recalling
$\pcf_{\com_\sigma(I)}({\ga}) =
\{\prod\limits_{i < \kappa}
\lambda_i,<_J:J \in \com_\sigma(I),
\lambda_i \in {\ga}$ for $i < \kappa\}$.
\end{enumerate}
\end{enumerate}
\end{claim}

\begin{remark}
1) This is \cite[6.13]{Sh:513}.

In a reasonable case the result
$\cov(|{\ga}|,\kappa^+,\kappa^+,\sigma)$.
\end{remark}

\begin{conclusion}
\label{cv.8}
In \ref{cv.7} if $\kappa < \mu_* \le \mu$ then
\mn
\begin{enumerate}
\item[$(a)$]  $\bold W_{I,\sigma}(\mu,< \mu_*)
= \sup\{\pp_{\com_\sigma(I)}(\mu)':
\mu_* \le \mu' \le \mu,\cf(\mu') \le \kappa\}$
\sn
\item[$(b)$]  if in (a) the left side is
a regular cardinal then the $\sup$ is obtained
for some sequence $\langle \lambda_i:i < \kappa
\rangle$ of regular cardinality and $J \in
\com_\sigma(I)$ such that $\lim_J\langle
\lambda_i:i < \kappa \rangle$ is well defined and
$\in [\mu_*,\mu]$ except possibly when
\sn
\begin{enumerate}
\item[$\circledast_{\mu,I,\sigma,\mu_*}$]  as in
$\circledast_{\mu,I,\sigma}$ above but $|{\ga}|
< \mu_*$.
\end{enumerate}
\end{enumerate}
\end{conclusion}

\begin{PROOF}{\ref{cv.7}}

\noindent
\underline{The inequality $\ge$}:

So assume $J$ is a $\sigma$-complete ideal
on $\kappa$ extending $I,\lambda_i \in \Reg
\cap \mu \backslash \kappa^+$ and $\mu =
\lim_J(\langle \lambda_i:i < \kappa \rangle)$
and $\lambda = \tcf(\prod\limits_{i < \kappa}
\lambda_i,<_J)$ is well
defined and we shall note that
$\bold W_{I,\sigma}(\mu) \ge \lambda$,
this clearly suffices, and let $\langle
f_\alpha:\alpha < \lambda
\rangle$ be $<_J$-increasing
cofinal in $(\prod\limits_{i < \kappa}
\lambda_i,<_J)$.  Now let $|{\cP}| <
\lambda,{\cP}$ be a family of sets of ordinals
each of cardinality $< \mu$.  For each $u \in
{\cP}$ let $g_u \in \prod\limits_{i < \kappa}
\lambda_i$ be defined
by $g_u(i) = \sup(u \cap \lambda_i)$ if
$|u| < \lambda_i$ and $g_u(i)= 0$ otherwise.

Hence for some $\alpha(u) < \lambda,g_u
<_J f_{\alpha(u)}$ and so
$\alpha(*) = \cup\{\alpha(u)+1:u \in
{\cP}\} < \lambda$ and $f_{\alpha(*)}$
exemplifies the failure of ${\cP}$ to exemplify
$\lambda > W_{I,\sigma}(\mu)$.
\medskip

\noindent
\underline{The inequality $\le$}:

Assume that $\lambda$ is regular $\ge
\pp^+_{I,\sigma}(\mu)$
and we shall prove that
$\bold W_{I,\sigma}(\mu) < \lambda$, this
clearly suffices.  Let $\chi$ be large
enough, and ${\gB}$ be an
elementary submodel of $({\cH}(\chi),
\in,<^*_\chi)$ of cardinality
$< \lambda$ such that $\{I,\sigma,
\mu,\lambda\} \subseteq {\gB}$
and $\lambda \cap {\gB}$ is an ordinal
which we shall call $\delta_{\gB}$.  Let
${\cP} =: [\mu]^{< \mu} \cap {\gB}$ so
$|{\cP}| < \lambda$.  Hence it is enough to
prove that $\bold W_{I,\sigma}(\mu) \le
|{\cP}|$ and for this it is enough to praove
that ${\cP}$ is as required in Definition
\ref{cv.2}(1).  Let
$\bar e = \langle e_\alpha:
\alpha < \mu \rangle \in {\gB}$ be such
that $e_\alpha$ is a club of
$\alpha$ of order type $\cf(\alpha)$ so
$e_{\alpha +1} = \{\alpha\},e_0 = \emptyset$.

So let $f_* \in {}^\kappa \mu$ and let $\langle
\mu_\varepsilon:\varepsilon < \cf
(\mu)\rangle \in {\gB}$ be an
increasing continuous sequence of
cardinals from $(\kappa,\mu)$ with
limit $\mu$.  Now by induction on
$n < \omega$ we choose
$\varepsilon_n,A_n,g_n,{\cT}_n,\bar S_n,
\bar B_n$ such that
\mn
\begin{enumerate}
\item[$\circledast_n$]  $(A)(a) \quad A_n
\in [\mu]^{\le \kappa},A_0 =
\{\mu_\varepsilon:\varepsilon < \cf(\mu)\}$
\sn
\item[${{}}$]  \hskip18pt $(b) \quad g_n$ is
a function from $\kappa$ to $A_n$
\sn
\item[${{}}$]  \hskip18pt $(c) \quad f_* \le g_n$
\sn
\item[${{}}$]  \hskip18pt $(d) \quad$ if
$n=m+1$ and $i <\kappa$ then
$g_m(i) > f_*(i) \Rightarrow g_n(i) > g_m(i)$
\sn
\item[${{}}$]  \hskip18pt $(e) \quad
{\cT}_n \subseteq {}^n \sigma$ has
cardinality $< \sigma$
\sn
\item[${{}}$]  \hskip18pt $(f) \quad {\cT}_0
= \{<>\}$
\sn
\item[${{}}$]  \hskip18pt $(g) \quad$ if
$n=m+1$ and $\eta \in {\cT}_n$
then $\eta \restriction m \in {\cT}_m$
\sn
\item[${{}}$]  \hskip18pt $(h) \quad
\bar S_n = \langle S_\eta:\eta \in {\cT}_n\rangle$
\sn
\item[${{}}$]  \hskip18pt $(i) \quad \bar B_n
= \langle B_\eta:\eta \in {\cT}_n\rangle$
\sn
\item[${{}}$]  \hskip18pt $(j) \quad
\varepsilon_n < \cf(\mu)$
and $n=m+1 \Rightarrow \varepsilon_n
\ge \varepsilon_m$
\sn
\item[${{}}$]  $(B) \qquad$ for each
$\eta \in {\cT}_n$:
\sn
\item[${{}}$]  \hskip18pt $(a) \quad
S_\eta \subseteq \kappa,S_\eta \notin {\cT}_n$
\sn
\item[${{}}$]  \hskip18pt $(b) \quad$ if
$n=m+1$ then
$S_{\eta \restriction m} \supseteq S_\eta$
\sn
\item[${{}}$]  \hskip18pt $(c) \quad B_\eta
\in {\gB}$ is a subset of $\mu$ of
cardinality $< \mu_{\varepsilon(n)}$
\sn
\item[${{}}$]  \hskip18pt $(d) \quad
\{g_n(i):i \in S_\eta\}$ is included in $B_\eta$
\sn
\item[${{}}$]  $(C)(a) \quad$ if $n=m+1$
and $\eta \in {\cT}_m$ then the set

\hskip30pt $S^*_\eta := \{i \in
S_\eta:g_m(i) > f_*(i)\}
\backslash \cup\{S_{\eta \char 94 <j>}:
\eta \char 94 \langle j\rangle \in {\cT}_n\}$

\hskip30pt belongs to $I$.
\end{enumerate}
\medskip

\noindent
\underline{It is enough to Carry the definition}:

Why?  As then $\{B_\eta:\eta \in {\cT}_n$ for
some $n < \omega\}$
is a family of members of ${\cP}$
(by (B)(c)), its cardinality is
$< \sigma$ (as $\sigma = \cf(\sigma) >
\aleph_0$ and for each
$n < \omega,|{\cT}_n| < \sigma$ by (A)(e)).

Similarly as $I$ is $\sigma$-complete the
set $S^* = \cup\{S^*_\eta:\eta \in {\cT}_n$
for some $n < \omega\}$ belongs to $I$.
Now for every $i \in \kappa \backslash S^*$,
we try to choose $\eta_n \in {\cT}_n$ by
induction on $n < \omega$ such that $i \in
S_{\eta_n}$ and $n = m+1 \Rightarrow \eta_m
= \eta_n \restriction m$
and $g_m(i) > f_*(i)$.  For $n=0$ let
$\eta = <>$ so $i \in \kappa = A_0$.
For $n=m+1$, as $i \notin S^*_{\eta_m}$,
see (C)(a) clearly $\eta_n$ as required
exists.  Now if $n=m+1$ again as $i \notin
S^*_{\eta_m}$ we get $g_m(i) > f_*(i)$ and
by (A)(d) we have $g_m(i) > g_n(i)$.  But there is no decreasing $\omega$-sequence of ordinals.
So for some $m,g_m(i) \le f_*(i)$ so by (A)(c), $g_m(i) = f_*(i)$ but $g_n(i) \in B_{\eta_n}$.
\medskip

\noindent
\underline{Carrying the induction}:
\smallskip

\noindent
\underline{Case $n=0$}:

Let ${\cT}_0 =\{<>\},A_{<>} =
\{\mu_\varepsilon:\varepsilon <
\text{\rm cf}(\mu)\}$ which has cardinality
$\le \kappa$ as $\cf(\mu) \le \kappa$ by
assumption.  Further, let $g_0$ be defined as the
function with domain $\kappa$ and $g_0(i) =
\min\{\mu_\varepsilon:\mu_\varepsilon >
f_*(i)\}$, let $S_{<>} =
\kappa$ and $B_{<>} =A_0$ which $\in {\gB}$
as $\langle \mu_\varepsilon:\varepsilon <
\cf(\mu)\rangle \in {\gB}$
(and has cardinality $|A_0| = \cf(\mu) \le
\kappa$).
\medskip

\noindent
\underline{Case $n=m+1$}:

Let $\eta \in {\cT}_m$ and define
$S'_\eta = \{i \in S_\eta:g_n(i) > f_*(i)\}$.
If $S'_\eta \in I$ then we decide that
$j < n \Rightarrow \eta {}^\frown \langle
j \rangle \notin {\cT}_n$, so we have nothing
more to do so assume $S'_\eta \notin I$.

Let ${\ga}_\eta = \{\cf(\alpha):\alpha \in
B_\eta$ and $\cf(\alpha) > |B_\eta| +
\kappa\}$ and let

\begin{equation*}
\begin{array}{clcr}
{\gc}_\eta = \{\tcf(\prod\limits_{i\in S'_\eta}
\cf(g_n(i)),<_J):&J \text{ is an }
\sigma\text{-complete ideal on} \\
  &S'_\eta \text{ extending } I
\restriction S'_\eta \text{ such that } \mu =
  \lim_J \langle \cf(g_n(i)):i \in
S'_\eta\rangle \\
  &\text{and } \prod\limits_{i \in S'_\eta}
\cf(g_n(i)),<_J) \text{   has true cofinality}\}
\end{array}
\end{equation*}

\mn
Clearly $\kappa + |{\ga}_\eta| <
\min({\ga}_\eta)$ and ${\gc}_\eta \subseteq
\pcf_{I,\sigma}({\ga}_\eta) \subseteq \lambda
\cap \Reg$ and by $\neg
\circledast_{\mu,I,\sigma}$ we know that
$\pcf_{I,\sigma}({\ga}_\eta)$ is a bounded
subset of $\lambda$.  But $B_\eta \in {\gB}$
hence ${\ga}_\eta \in {\gB}$ hence
$\pcf_{I,\sigma}(\ga_\eta) \in \gB$ so
as $\gB \cap \lambda = \delta_{\gB} <
\lambda$, clearly $\pcf_{I,\sigma}({\ga}_\eta)
\subseteq {\gB}$ hence $\theta \in {\gc}_\eta
\Rightarrow \theta < \delta_{\gB}$.
Using $\pcf$ basic properties
let $J_{\eta,\lambda}$ be
the $\sigma$-complete ideal on
${\ga}_\eta$ generated by
$J_{=\lambda}[{\ga}_\eta]$ and so
$\bar{\ga}_\eta,J_{\eta,\lambda} \in {\gB}$
and there is a $<_{J_{\eta,\lambda}}$-increasing
cofinal sequence $\bar f_{\eta,\lambda} =
\langle f_{\eta,\lambda,\zeta}:\zeta <
\lambda \rangle$ of members of
$\Pi {\ga}_\eta$ such that
$f_{\eta,\lambda,\zeta}$ is the
$<_{J_{\eta,\lambda}}$-e.u.b. of
$\bar f_{\eta,\lambda} \restriction
\zeta$ when there is such $<_{J_{\eta,
\lambda}}$-e.u.b.  Without loss of
generality $\bar f_{\eta,\lambda} \in \gB$ hence
$\{f_{\eta,\lambda,\zeta}:\zeta < \lambda\}
\subseteq {\gB}$.

Let ${\ga}_m = \cup\{{\ga}_\eta:\eta \in
{\cT}_m\}$ and define a
$h_m \in \Pi{\ga}_m$ by $h_m(\theta) =
\sup\{\otp(e_{g_m(i)} \cap f_*(i)):i < \kappa$
and $f_*(i) < g_m(i)\}$.
Clearly it is $< \theta$ as $\theta = \cf(\theta) >
\mu_{\varepsilon(m)} \ge |B_\eta| + \kappa$
when $\theta \in {\ga}_\eta$.
For each $\eta \in {\cT}_m$ and $\lambda \in
{\gc}_\eta$ let $\zeta_{\eta,\lambda} <
\lambda$ be such that $h_m \restriction \ga_\eta <
f_{\eta,\lambda,\zeta_{\eta,\lambda}}
\mod J_{\eta,\lambda}$, and let

\[
S^1_{\eta,\lambda} = \{i \in S_\eta:
h_m(\cf(g_i(\theta)) <
f_{\eta,\lambda,\zeta_{\eta,\lambda}}(\cf(g_m(i))\}
\]

\mn
\begin{enumerate}
\item[$\bigodot$]  for some subset $\gc'_\eta$
of ${\gc}_\eta$ of cardinality $< \sigma$ the
set $\{i \in S_\eta:i \notin S^1_{\eta,\lambda}$
for every $\lambda \in {\gC}'_\eta\}$ belongs
to $I$.
\end{enumerate}
\mn
[Why?  Otherwise, let $J$ be the
$\sigma$-complete ideal on $S_\eta$
generated by $I \cup \{S^1_{\eta,\lambda}:
\lambda \in {\gc}_\eta\}$, so $\kappa
\notin J$ hence for some $S^* \in J^+$ we know
that $(\prod\limits_{i \in S^*} \cf(g_m(i),
<_{J \restriction S^*})$ has true
cofinaltiy, call it $\lambda^*$.
Necessarily $\lambda^* \in {\gc}_\eta$ and
easily get a contradiction.]
\bigskip

\noindent
\underline{Case A}:   $|\cup\{\gc_\eta:\eta
\in \cT_m\}| < \mu$.

Let $\langle \lambda_{\eta,j}:j < j_\eta
\rangle$ list ${\gc}'_\eta$.
Let ${\ga}'_n = {\ga}_n \backslash
|\bigcup\limits_\eta {\gc}_\eta|^+$.  Now
by induction on $k< \omega$ we choose
$h_{n,k},\zeta_{\eta,j,k}$ for $j < j_\eta,
\eta \in \cT_m$ such that
\mn
\begin{enumerate}
\item[$\circledast$]  $(a) \quad h_{m,k} \in
\Pi{\ga}'_m$
\sn
\item[${{}}$]  $(b) \quad h_{m,k} < h_{m,k+1}$
\sn
\item[${{}}$]  $(c) \quad h_{m,0} = h_m$
\sn
\item[${{}}$]  $(d) \quad \zeta_{\eta,j,k}
< \lambda_{\eta,j}$
\sn
\item[${{}}$]  $(e) \quad \zeta_{\eta,j,k}
< \zeta_{\eta,j,k+1}$
\sn
\item[${{}}$]  $(f) \quad \zeta_{\eta,j,0}
= \zeta_{\eta,j}$
\sn
\item[${{}}$]  $(g) \quad h_{m,k+1}(\theta) =
\sup[\{f_{\eta,\lambda_{\eta,j,
\zeta_{\eta,j,k}}}(\theta):\eta \in
{\cT}_n,\theta \in {\ga}_\eta\} \cup
\{h_{m,k}(\theta)\}]$
\sn
\item[${{}}$]  $(h) \quad \zeta_{\eta,j,k+1} =
\Min\{\zeta < \lambda_{\eta,j}:\zeta >
\zeta_{\eta,j,k}$ and
$h_{m,k+1} \restriction {\ga}_\eta <
f_{\eta,\lambda_{\eta,j,\zeta}} \mod$

\hskip25pt $J_{\eta,\lambda_{\eta,j}}\}$.
\end{enumerate}
\mn
There is no problem to carry the induction.
Let $h_{m,\omega} \in \Pi{\ga}_m$ be defined
by $h_{m,\omega}(\theta) = \cup\{h_{m,k}(\theta):
k < \omega\}$.  Let $S'_{\eta,j} = \{i \in
S_\eta:f_*(i)$ is $<$ the
$h_{m,\omega}(\cf(g_m(i))$-ith member of
$e_{g_m(i)}\}$.

Now
\mn
\begin{enumerate}
\item[$\boxtimes$]  for some ${\gc}''_\eta
\subseteq {\gc}_\eta,|{\gc}''_\eta| < \sigma$
for $\eta \in {\cT}_m$ we have
$S_n \backslash \cup \{S_{\eta,j}:
\lambda_j \in {\gc}'_\eta\} \in I$.
\end{enumerate}
Now continue.
\end{PROOF}
\bigskip

\noindent
\underline{Case B}:  $C$ not Case A.

Use \S2.
\bigskip

\centerline{$* \qquad * \qquad *$}
\bigskip

\begin{discussion}
\label{cv.10}
Lemma \ref{cv.7} leaves us in a
strange situation: clause (a) is fine,
but concerning the exception in clause (b);
it may well be impossible and $\pcf({\ga})$ is
always not ``so large".  We do not know this,
we try to clarify the case for
reasonable $\bold J_i$, i.e.,
closed under products of two.
\end{discussion}

\begin{observation}
\label{cv.11}
1) There is $\mu_* < \mu$ such that
$(\forall \mu')(\mu_* < \mu' \le \mu
\wedge \cf(\mu') \le \kappa < \mu')
\Rightarrow \pp^+_{\bold J}(\mu') \le
\text{\rm pp}^+_{\bold J}(\mu)$ when:
\mn
\begin{enumerate}
\item[$\circledast$]  $(a) \quad \cf(\mu) \le
\kappa < \mu$
\sn
\item[${{}}$]  $(b) \quad \bold J$ is a set
of $\sigma$-complete ideals
\sn
\item[${{}}$]  $(c) \quad J \in \bold J
\Rightarrow |\Dom(J)| \le \kappa$
\sn
\item[${{}}$]  $(d) \quad$ if $J_\varepsilon
\in \bold J$ for $\varepsilon < \cf(\mu)$
then for some $\sigma$-complete ideal $I$ on
$\cf(\mu)$,

\hskip25pt  the ideal $J = \Sigma_I\langle
J_\varepsilon:\varepsilon < \cf(\mu)\rangle$
belongs to $\bold J$ (or is

\hskip25pt  just $\le_{\RK}$ from some
$J' \in \bold J$).
\end{enumerate}
\end{observation}

\begin{proof}{\ref{cv.11}}
Let $\Lambda = \{\mu':\mu'$ is a
cardinal $< \mu$ but
$> \kappa$, of cofinality $\le \kappa$ such
that $\pp^+_{\bold J}(\mu') >
\pp_{\bold J}(\mu)\}$, and assume toward
contradiction that $\mu = \sup(\Lambda)$.  So
we can choose an increasing sequence
$\langle \mu_\varepsilon:\varepsilon <
\cf(\mu)\rangle$ of members of $\Lambda$
with limit $\mu$.  For each $\varepsilon <
\cf(\mu)$ let $J_\varepsilon \in \bold J$ witnesses
$\mu_\varepsilon \in \Lambda$.   Without
loss of generality $\kappa_\varepsilon =
\Dom(J) \le \kappa$ so we can find
$\langle \lambda_{\varepsilon,i}:i <
\kappa_\varepsilon\rangle$ witnessing this.
In particular
$(\prod\limits_{i < \kappa_\varepsilon}
\lambda_{\varepsilon,i},<_{J_\varepsilon}))$ has
true cofinality $\lambda_\varepsilon =
\cf(\lambda_\varepsilon) \ge
\pp^+_{\bold J}(\mu)$.  Let
$I,J$ be as in clause (d) of $\circledast$.
\end{proof}
\bigskip

\centerline{$ * \qquad * \qquad$}
\bigskip

A dual kind of measure to Definition \ref{cv.1} is
\begin{definition}
\label{cv.21}
1) Assume $J$ is an ideal say on
$\kappa$ and $f^*:\kappa \rightarrow \Ord$
and $\mu$ cardinal.  Then $\bold U_J(f^*,< \mu)
= \Min\{|\cP|:\cP$ a family of subsets of
$\sup \Rang(f) +1$ each of cardinality $< \mu$
such that for every $f \le f^*$ (i.e., $f \in
\prod\limits_{i < \kappa} (f^*(i)+1))$
there is $A \in {\cP}$ such that
$\{i < \kappa:f(i) \in A\} \notin J\}$.

\noindent
2) If above we write $\bold J$ instead of
$J$ this means $\bold J$ is
a family of ideals on $\kappa$ and the $\cP$
should serve all the $J \in \bold J$
simultaneously.
\end{definition}

\begin{claim}
\label{cv.22}
We have $\bold U_{J^{\bd}_\kappa}(\mu,< \mu) =
\lambda_*$ if we assume
\mn
\begin{enumerate}
\item[$\circledast$]  $(a) \quad \mu > \kappa
= \cf(\mu) > \aleph_0$
\sn
\item[${{}}$]  $(b) \quad ([\kappa]^\kappa,
\supseteq)$ satisfies the $\mu$-c.c. or just
$\mu^+$-c.c. which means that:

\hskip25pt if ${\cA} \subseteq
[\kappa]^\kappa$ and $A \ne B \in {\cA}
\Rightarrow |A \cap B| <
\kappa$ then $|{\cA}| \le \mu$
\sn
\item[${{}}$]  $(c) \quad \lambda_* =
\pp_{J^{\bd}_\kappa}(\mu) =
\sup\{\tcf(\prod\limits_{i < \kappa}
\lambda_i,<_{J^{\bd}_\kappa}):\lambda_i <\mu$ is
increasing with

\hskip25pt limit $\mu$ and
$(\prod\limits_{i < \kappa} \lambda_i,
<_{J^{\bd}_\kappa})$ has true cofinality$\}$.
\end{enumerate}
\end{claim}

\begin{claim}
\label{cv.23}
We can in \ref{cv.22} replace
$J^{\bd}_\kappa$ by any $\aleph_1$-complete
filter $J$ (?) on $\kappa$ (so (b) becomes
``$(J^+,\supseteq)$ satisfies the $\mu^+$-c.c.''
\end{claim}

\begin{remark}
\label{cv.24}
If in clause (b) of $\otimes$ of
\ref{cv.22}, we use the $\mu$-c.c. the proof
is simpler, using ${\cT}_n \subseteq
{}^n(\mu_{\varepsilon_n}),\varepsilon_n \le
\varepsilon_{n+1}$.
\end{remark}

\begin{proof}
Let
\mn
\begin{enumerate}
\item[$(*)$]  $(a) \quad \bar\mu =
\langle \mu_i:i < \kappa \rangle$ is an
increasing continuous sequence of singular

\hskip25pt  cardinals $> \kappa$ with limit $\mu$.
\end{enumerate}
\mn
Let $\chi$ be large enough, $<^*_\chi$ a
well ordering of $({\cH}(\chi),\in)$ and
${\cB}$ an elementary submodel of
$({\cH}(\chi),\in,<^*_\chi)$ of cardinality
$\lambda_*$ such that $\lambda_*
+1 \subseteq {gB}$ and $\bar \mu \in {\gB}$
and let ${\cA} = [\mu]^{<\mu} \cap {\gB}$.

So ${\cA}$ is a family of sets of the right
form and has cardinality $\le \lambda_*$.
It remains to prove the major point:
assume $S$ is an unbounded subset of
$\kappa,f^* \in \prod\limits_{i \in S}
[\mu_i,\mu_{i+1}]$ we should prove that
$(\exists A \in {\cA})(\exists^\kappa i \in S)
(f(i) \in A)$.

Let $\bar e = \langle e_\alpha:\alpha < \mu
\rangle \in {\gB}$ be such that $e_\alpha$
is a club of $\alpha$ of order type $\cf(\alpha)$
so $e_{\alpha +1} = \{\alpha\},e_0 =
\emptyset$.  Let $\langle
\beta_{\alpha,\varepsilon}:
\varepsilon < \cf(\alpha)\rangle$ be
an increasing enumeration of $e_\alpha$.

We choose $\varepsilon_n,g_n,A_n,I_n,
\langle S_\eta,B_\eta:\eta \in {\cT}_n\rangle$
such that
\mn
\begin{enumerate}
\item[$\circledast_n$]  $(A)(a) \quad
{\cT}_n \subseteq {}^n \mu,{\cT}_0 =
\{<>\},[n = m+1 \wedge \eta \in
{\cT}_n \Rightarrow \eta
\restriction m \in {\cT}_n]$
\sn
\item[${{}}$]  \hskip18pt $(b) \quad A_n
\subseteq \mu$ has cardinality $\le \kappa$
\sn
\item[${{}}$]  \hskip18pt $(c) \quad g_n:
\kappa \rightarrow A_n$
\sn
\item[${{}}$]  \hskip18pt $(d) \quad i <
\kappa \Rightarrow f^*(i) \le g_n(i)$
\sn
\item[${{}}$]  \hskip18pt $(e) \quad n = m+1
\Rightarrow g_n \le g_m$
\sn
\item[${{}}$]  \hskip18pt $(f) \quad
\varepsilon_n < \kappa$ and $n=m+1
\Rightarrow \varepsilon_m < \varepsilon_n$
\sn
\item[${{}}$]  \hskip18pt $(g) \quad$ if
$n=m+1,i \in (\varepsilon_n,\kappa)$ and
$g_m(i) > f^*(i)$ then $g_m(i) > g_n(i)$
\sn
\item[${{}}$]  $(B)$ for $\eta \in {\cT}_n$
\sn
\begin{enumerate}
\item[${{}}$]  $(a) \quad S_\eta \subseteq
\kappa$ has cardinality $\kappa$
\sn
\item[${{}}$]  $(b) \quad S_\eta
\in [\kappa]^\kappa$ and $\nu
\triangleleft \eta \Rightarrow S_\eta
\subseteq S_\nu$
\sn
\item[${{}}$]  $(c) \quad B_\eta \in
{\gB}$ is a subset of $\mu$ of
cardinality $< \mu_{\varepsilon(n)}$ where
$\varepsilon(n) =$

\hskip45pt $\Min\{\varepsilon < \kappa:\eta
\in {}^n(\mu_\varepsilon)$ and
$\varepsilon \ge \varepsilon_n\}$
\sn
\item[${{}}$] \hskip18pt $(d) \quad \{g_n(i):
i \in S_\eta\} \subseteq B_\eta$.
\end{enumerate}
\end{enumerate}
\mn
For $n=0$ let $\varepsilon_0=0,A_{<>} =
\{\mu_i:i < \kappa\},{\cT}_0 =
\{<>\},S_{<>} = \kappa,g_m$ is the function
with domain $\kappa$ such that
$g_{<>} = \Min\{\alpha \in A_{<>}:f^*(i) <
\alpha\}$.
Assume $n=m+1$ and we have defined for $m$.

Let

\begin{equation*}
\begin{array}{clcr}
{\gc}_n = \bigl\{\theta:&\text{ there is an
increasing sequence }
\langle \lambda_i:i < \kappa \rangle \\
  &\text{ of regular cardinals } \in
(\kappa,\mu) \text{ with limit }
  \mu \text{ such that} \\
  &\theta = \tcf(\prod\limits_{i < \kappa}
  \lambda_i,<_{J^{\bd}_\kappa}) \text{ and} \\
  &\{\lambda_i:i < \kappa\} \subseteq
\{\cf(\alpha):\alpha \in
  A_m,\cf(\alpha) > \kappa\bigr\}.
\end{array}
\end{equation*}

\mn
Of course, ${\gc}_n \subseteq \Reg \backslash
\mu$.  Now for each $\theta \in {\gc}_n$ let
$\langle \lambda^\theta_i:i <
\kappa \rangle$ exemplifies it
so $\{\{\lambda^\theta_i:i <
\kappa\}:\theta \in {\gc}_n\}$ is a family
of subsets of $\{\cf(\alpha):\alpha \in A_m,
\cf(\alpha) > \kappa)\}$ each of cardinality
$\kappa$ and the intersection of any two has
cardinality $< \kappa$.

As $|A_m| \le \kappa$, by assumption (d) of
the claim we know that $|{\gc}_n| \le \mu$
and let $\langle \lambda_\beta:\beta \le
\mu \rangle$ list them.

For each $\eta \in {\cT}_m$ and
$\varepsilon < \kappa$ let

\[
{\ga}_{\eta,\varepsilon} = \{\cf(\delta):\delta
\in B_\eta \text{ and cf}(\delta) >
\mu_\varepsilon + |B_\eta|\}
\]

\mn
so

\[
|{\ga}_{\eta,\varepsilon}| \le |B_\eta| <
\min({\ga}_\eta).
\]

\mn
Let $W = \{(\eta,\varepsilon,\beta):\eta \in
{\cT}_m,\varepsilon <
\kappa,\beta < \mu_\varepsilon\}$.  Clearly
${\ga}_{\eta,\varepsilon} \in {\gB},
\lambda_\beta \in {\gB}$ hence
$J_{\eta,\varepsilon,\beta} =$ the
$\kappa$-complete ideal generated by
$J_{=\lambda_\beta}
[{\ga}_{\eta,\varepsilon}]$ belongs to ${\gB}$
and some $<_{J_{\eta,\varepsilon,\beta}}$-increasing and cofinal sequence
$\langle f_{\eta,\varepsilon,\beta,\zeta}:\zeta <
\lambda_\beta\rangle$ belongs to ${\gB}$ and
$f_{\eta,\varepsilon,\beta,\zeta}$ is an
$<_{J_{\eta,\varepsilon,\beta}}$-e.u.b. of $\langle
f_{\eta,\varepsilon,\beta,\xi}:\xi < \zeta \rangle$ when there is one.

We now define a function $h_m$

\[
\Dom(h_m) = \ga^*_m = \cup\{\ga_{\eta,\varepsilon}:
\eta \in {\cT}_m \text{ and } \varepsilon <
\kappa\}
\]

\mn
so

\[
\theta \in \Dom(h_m) \Rightarrow \kappa <
\theta < \mu \wedge \theta \in \Reg
\]

\mn
(in fact we do not exclude the case ${\ga}^*_m
= \Reg \cap \mu \backslash \kappa^+$) and

\[
h_m(\theta) = \sup\{e_{g_n(i)} \cap
f*(i):i < \kappa \text{ and } \cf(g_n(i))=\theta\}.
\]

\mn
As $\theta = \cf(\theta) > \kappa$ clearly

\[
\theta \in \Dom(h_m) \Rightarrow h_m(\theta)
< \theta.
\]

\mn
We choose now by induction on
$k<\omega,h_{m,k},\langle
\zeta^k_{\eta,\varepsilon,\beta}:
(\eta,\varepsilon,\beta) \in W \rangle$ such that
\mn
\begin{enumerate}
\item[$\boxtimes$]  $(a) \quad h_{m,k} \in
\Pi{\ga}^*_m$
\sn
\item[${{}}$]  $(b) \quad h_{m,0} = h_m$
\sn
\item[${{}}$]  $(c) \quad h_{m,k} \le h_{m,k+1}$
\sn
\item[${{}}$]  $(d) \quad
\zeta^k_{\eta,\varepsilon,\beta} =
\Min\{\zeta:h_{m,k} \restriction
\ga_{\eta,\varepsilon}
<_{J_{\eta,\varepsilon,\beta}}
f_{\eta,\varepsilon,\beta,\zeta}$
and $\ell < k \Rightarrow
\zeta^\ell_{\eta,\varepsilon,\beta} < \zeta\}$
\sn
\item[${{}}$]  $(e) \quad h_{m,k+1}(\theta) =
\sup[\{h_{m,k}(\theta)\} \cup
\{f^k_{\eta,\beta,\varepsilon,
\zeta^k_{\eta,\varepsilon,\eta}}(\theta)$:
the triple $(\eta,\beta,\varepsilon) \in W$

\hskip25pt  satisfies
$(\exists \varepsilon)(\beta < \mu_\varepsilon <
\theta)$ and $\theta \in
\ga_{\eta,\varepsilon}\}]$.
\end{enumerate}
\mn
Note that $h_{m,k+1}(\theta) < \theta$ as the
$\sup$ is over a set of $< \theta$ ordinals.

So we have carried the definition, and let
$h^*_{m,w} \in \Pi{\ga}_m$ be defined by
$h_{m,\omega}(\theta) = \sup\{h_{m,k}(\theta):k <
\omega\}$ and $\zeta_{\eta,\varepsilon,\beta} =
\zeta(\eta,\varepsilon,\beta) =
\sup\{\zeta^k_{\eta,\varepsilon,\beta}:
k < \omega\}$.  Now for each
$(\eta,\varepsilon,\beta) \in W$ we have
$k < \omega \Rightarrow h_{m,k}
\restriction {\ga}_{\eta,\varepsilon}
<_{J_{\eta,\varepsilon,\beta}}
f^k_{\eta,\varepsilon,\beta,
\zeta(\eta,\varepsilon,\beta)}) <
h_{m,k+1} \restriction
{\ga}_{\eta,\varepsilon}$.  By the choice
of $\bar f_{\eta,\varepsilon,\beta}$ as
$J_{\eta,\varepsilon,\beta}$
is $\aleph_1$-complete it follows that
$h_{m,w} \restriction {\ga}_{\eta,\varepsilon} =
f_{\eta,\varepsilon,\beta,
\zeta_{\eta,\varepsilon,\beta}} \mod
J_{\eta,\varepsilon,\beta}$.

Let

\begin{equation*}
\begin{array}{clcr}
A_n =: \{\alpha':& \text{for some } \alpha
\in A_n,\cf(\alpha)
\in {\ga}_n \text{ and } \alpha' \\
  &\text{ is the } h_{m,\omega}
(\cf(\alpha))\text{-th member of } e_\alpha\}.
\end{array}
\end{equation*}

\begin{equation*}
\begin{array}{clcr}
g_n(i) &\text{ is } \alpha' \text{ when }
\alpha' \text{ is the }
h_{m,\omega}(\cf(g_m(i))\text{-th member of} \\
  &e_{g_m(i)} \text{ and zero otherwise}.
\end{array}
\end{equation*}

\mn
The main point is why $\sigma_n \in
(\varepsilon_m,\kappa)$ exists.

To finish the induction step on $n$, let

\[
B_{\eta,\varepsilon,\beta} = \Rang
(f_{\eta,\varepsilon,\eta,\zeta_{\eta,
\varepsilon,\beta}})
\]

\[
B'_{\eta,\varepsilon} =
B_{\eta,\varepsilon,\beta} \cup
\{e_\alpha:\alpha \in B_{\eta,\varepsilon}
\text{ and } \cf(\alpha) \le
\mu_{\varepsilon(n)}\}
\]

\mn
and we choose $\langle B_\rho:\rho \in \cT_n,
\rho \restriction m \in B =
\eta$ to list them enumerates
$\{B_{\eta,\varepsilon,\beta}:
\varepsilon,\beta\}$ are such that
$(\eta,\varepsilon,\beta) \in W_m
\cup \{B'_{\eta,\varepsilon}\}$ in a
way consistent with the induction hypothesis.

Having carried the induction on $n$, note that
\mn
\begin{enumerate}
\item[$\circledast_1$]  for some
$n,u_n = \{i < \kappa:f^*(i) =
g_n(i)\} \in [\kappa]^\kappa$

\end{enumerate}
\mn
We now choose by induction on $m \le n$ a
sequence $\eta_m \in {\cT}_m$ such that
$\eta_0 = <>,m = \ell +1 \Rightarrow \eta_\ell
\triangleleft \eta_m$ and $S_\eta \cap u_n \in
[\kappa]^\kappa$.  For $m=n$ by
\mn
\begin{enumerate}
\item[$\circledast(*)$]  $u' = u \cap
S_{\eta_n} \in [\kappa]^\kappa$ and $\Rang(f^*
\cap u') \subseteq B_\eta \in {\cP}$
so we are done.
\end{enumerate}
\end{proof}

\begin{discussion}
\label{cv.27}
1) Can we consider ``$\bold c([\mu]^\mu,
\supseteq) \le \mu^+$"?  We should look
again at \S2.

\noindent
2) More hopeful is to replace
$\bold U_{J^{\bd}_\kappa}(\mu)$ by
$\bold U_{\text{non-stationary}_\kappa}(\mu)$.

\noindent
3) By \ref{cv.11} and \ref{cv.21} we should
have the $\prd$ version
(for which $\bold J$ and closure, see
\cite{Sh:410}.
\end{discussion}
\newpage

\bibliographystyle{alphacolon}
\bibliography{lista,listb,listx,listf,liste,listz}

\end{document}